\documentclass[letterpaper,11pt,reqno]{amsart}

\makeatletter
\usepackage{amssymb}
\usepackage{latexsym}
\usepackage{amsbsy}
\usepackage{amsfonts}
\usepackage{graphicx}
\usepackage{enumerate}
\usepackage{enumitem}
\usepackage{mathtools}
\usepackage{color}
\usepackage{url}
\usepackage{comment}
\usepackage{appendix}

\usepackage{tikz}

\def\marginpar#1{\ignorespaces}

  \newcommand{\beq}{\begin{equation}}
    \newcommand{\eeq}{\end{equation}}
    
    \newcommand{\bal}{\begin{align}}
    \newcommand{\eal}{\end{align}}
    \newcommand{\bals}{\begin{align*}}
    \newcommand{\eals}{\end{align*}}
    
    \newcommand{\calA}{{\mathcal A}}
    
    \newcommand{\bbR}{{\mathbb{R}}}

    \newcommand{\bbZ}{{\mathbb{Z}}}

    \newcommand{\bbT}{{\mathbb{T}}}

    \newcommand{\calB}{{\mathcal B}}

    \newcommand{\calC}{{\mathcal C}}

    \newcommand{\calK}{{\mathcal K}}

    \newcommand{\calP}{{\mathcal P}}

\newcommand{\bd}{\mathbf d}
    
    \newcommand{\eps}{\varepsilon}
    
    \newcommand{\lb}{\label}

\DeclareMathOperator{\tr}{Tr}

\newtheorem{theorem}{Theorem}[section]
\newtheorem{remark}{Remark}[section]
\newtheorem{lemma}[theorem]{Lemma}
\newtheorem{proposition}[theorem]{Proposition}
\newtheorem{corollary}[theorem]{Corollary}
\newtheorem{definition}[theorem]{Definition}

\numberwithin{equation}{section}
\makeatother
\begin{document}
\title[Convergence Rate of Vanishing Viscosity for MFGs]{The Convergence Rate of Vanishing Viscosity Approximations for Mean Field Games}

\author[W. Tang, Y. P. Zhang]{Wenpin Tang, Yuming Paul Zhang}

\address[W. Tang]{Department of Industrial Engineering and Operations Research, Columbia University, S.W. Mudd Building, 
500 W 120th St, New York, NY 10027} 
\email{wt2319@columbia.edu}

\address[Y. P. Zhang]
{Department of Mathematics and Statistics, Auburn University, Parker Hall, 
221 Roosevelt Concourse, Auburn, AL 36849}
\email{yzhangpaul@auburn.edu}

\date{\today} 


\begin{abstract}
Motivated by numerical challenges in first-order mean field games (MFGs) and the weak noise theory for the Kardar–Parisi–Zhang equation,
we consider the problem of vanishing viscosity approximations for MFGs. 
We provide the first results on the convergence rate to the vanishing viscosity limit in mean field games,
with a focus on the dimension dependence of the rate exponent.
Two cases are studied: MFGs with a local coupling and those with a nonlocal, regularizing coupling. 
In the former case, we use a duality approach
and our results suggest that there may be a phase transition in the dimension dependence of vanishing viscosity approximations 
in terms of the growth of the Hamiltonian and the local coupling. 
In the latter case, we rely on the regularity analysis of the solution, 
and derive a faster rate compared to MFGs with a local coupling.
A list of open problems are presented. 
\end{abstract}
\maketitle


\smallskip
\maketitle


\section{Introduction}
\label{sc1}

\quad Mean Field Games (MFGs) are a mathematical framework used to model and analyze strategic interactions in a large population,
which was independently developed by Lasry and Lions \cite{LL06a, LL06b, LL07}, and Caines, Huang and Malham\'e \cite{HMC}.
In MFGs, each individual makes decisions based on her own objective
as well as the behavior of the entire population,
represented by a mean field which describes the probability distribution of the collective state of the system. 
MFGs are widely used to model complex systems in economics \cite{carmona, jovanovic} and engineering \cite{djehiche, huang2007large, HMC}.

\quad The standard form of MFGs is given by the following system of partial differential equations (PDEs):
\beq\lb{1.1}
\left\{
\begin{aligned}
&-\partial_t u_\nu-\nu\Delta u_\nu+H(x,Du_\nu)=f(x,m_\nu),\\
&\partial_t m_\nu-\nu\Delta m_\nu-\nabla\cdot(m_\nu D_pH(x,Du_\nu))=0, \qquad \mbox{in } \Omega:=(0,T)\times \bbT^d, \\
&m_\nu(0,x)=\bar m(x),\quad u_\nu(T,x)=\bar u(x),
\end{aligned}
\right.
\eeq
where $\bbT^d:=\bbR^d/\bbZ^d$ is the $d$-dimensional torus,
$T > 0$ and $\nu\geq 0$.
The Hamiltonian $H(x,p)$ is a convex function with respect to the second variable $p$.
MFGs are used to describe Nash equilibria in differential games with a continuum of players,
where $m_\nu(t,x)$ is the density of players at time $t$ and at position $x$.
The variable $u_\nu$ is the value of a typical player's optimal control problem,
so it is a solution to some Hamilton-Jacobi equation.
As $u_\nu$ is the optimal value (that the player can possibly achieve), 
the optimal strategy is $-D_pH(x,Du_\nu)$.
If the density of the population flows to the direction preferred by the optimal strategy, 
the game has a Nash equilibrium. 
This amounts to solving \eqref{1.1}.
When $\nu > 0$, the system is of second order. 
The system by sending $\nu$ to zero is called the {\em vanishing viscosity limit}, which is of first order.

\quad  In this paper, we study the convergence rate of second-order MFGs to the vanishing viscosity limit as $\nu \rightarrow 0^+$.
A special focus is on the dependence of the rate on the dimension $d$, 
and hence on whether or under what circumstances it induces the curse of dimensionality or the lack thereof.
For simplicity, we assume that the terminal data $\bar{u}$ is independent of the density $m_\nu$.
We distinguish two cases for the coupling (or the running cost) $f$:
\begin{itemize}[itemsep = 3 pt]
\item
When $f(x,m)$ depends on the pointwise value of the density $m(t,x)$,
the coupling is referred to as {\em local}. 
\item
When $f(x,m)$ depends on the entire distribution of $m(t, \cdot)$ and is uniformly smooth for all distributions,
it is referred to as a {\em regularizing, nonlocal coupling}. 
\end{itemize}
We assume that $f$ is increasing in $m$ in the local case,
and satisfies the Lasry-Lions monotonicity condition in the nonlocal case, so that the uniqueness of solutions is guaranteed.
It is also worth mentioning that the convergence rate of vanishing viscosity approximations
of Hamilton-Jacobi equations was studied in \cite{two, evans2010adjoint, tang2023policy, tran2011adjoint},
and the optimal rate is $\nu^{\frac{1}{2}}$.
The same problem was considered for hyperbolic systems in \cite{BB05, BHY12, BY04},
and for Fokker-Planck equations with $\mathcal{C}^1$ nonlocal drifts in \cite{fetecau2022zero, fetecau2019swarming, zhang}. 

\quad Before delving into the problem, 
we digress a bit to explain the motivations to study the convergence of vanishing viscosity in MFGs. 
 \begin{enumerate}[itemsep = 3 pt]
\item
In recent years, there has been a growing interest in modeling autonomous vehicles' control and their macroscopic traffic flow
by first-order MFGs ($\nu = 0$) \cite{HDDC20b, HDDC20, KAS16}.
As pointed out in \cite{HDDC20b}, numerical methods converge slowly, or even fail to converge for first-order MFGs.
This is not surprising as iterative algorithms may be ill-posed due to irregular coefficients in the transport equation.
On the other hand, it is known \cite{CCG21, C22, CT22} that
the policy iteration algorithm converges exponentially fast for second-order MFGs ($\nu > 0$).
So a reasonable idea is to approximate first-order MFGs by second-order MFGs,
and a quantitative rate of second-order MFGs to the vanishing viscosity limit provides the approximation error.
Moreover, the convergence of the policy iteration algorithm is exponential in $\nu^{-1}$, 
which yields a tradeoff between bias and algorithm efficiency.
\item
The Kardar–Parisi–Zhang (KPZ) universality class describes the limiting behavior of a collection of random growth models,
and the underlying continuum object is the KPZ equation \cite{Corwin12, Quas12}. 
Due to its nonlinear nature, the KPZ equation is hardly accessible 
except for some initial conditions which lead to integrability.
Recently, there has been a line of work on large deviations of the stochastic heat equation (SHE),
and hence the ($1+1$)-dimensional KPZ equation by the Cole-Hopf transform,
under the {\em weak noise theory} \cite{KMS16, MKV16}.
Rigorous treatments have been developed in \cite{GLT21, LT21st, LT22, Tsai22}.
Under narrow wedge initial condition, the ``most probable'' KPZ path 
conditioned to be $\lambda$ at time $T$ is given by $h(t,x) = \log Z[\rho^{\texttt{m}}](t,x)$, 
where $Z[\rho]$ given $\rho = \rho(t,x)$ solves the PDE
\begin{equation*}
\partial_t Z = \frac{1}{2} \partial_{xx} Z + \rho Z, \mbox{ for } (t,x) \in (0,T] \times \mathbb{R}, \quad Z(0, \cdot) = \delta_0,
\end{equation*}
and $\rho^{\texttt{m}}$ solves the variational problem
$\inf \left\{\frac{1}{2} \|\rho\|_{L^2}^2: Z[\rho](T,0) = e^{\lambda}\right\}$.
Of particular interest is the lower-tail limit as $\lambda:= -\nu^{-1} \to -\infty$ (so $\nu \to 0^+$).
Setting the scaling
$Z_\nu[\rho](t,x):= Z[\nu^{-1} \rho(\cdot, \, \nu^{\frac{1}{2}} \cdot)](t, \, \nu^{-\frac{1}{2}}x)$, 
$\rho_\nu(t,x):= \nu \rho^{\texttt{m}}(t, \, \nu^{-\frac{1}{2}}x)$,
and $h_\nu(t,x):= \nu^{-1} \log Z_\nu[\rho_\nu](t,x)$,
\cite{LT22, Tsai22} showed that $h_\nu$ converges locally uniformly to a limiting shape $h_*$  as $\nu \to 0^+$
(by integrability of $h_*$) but with no rate. 
Curiously, $(h_\nu, \rho_\nu)$ solves the PDEs:
\begin{equation}
\lb{1.2}
\left\{
\begin{aligned}
&\partial_t h_\nu=\frac{\nu}{2} \partial_{xx} h_\nu+\frac12 (\partial_x h_\nu)^2+\rho_\nu,\\
&-\partial_t \rho_\nu= \frac{\nu}{2} \partial_{xx} \rho_\nu -\nabla\cdot(\rho_\nu \, \partial_x h_\nu),
\end{aligned}
\right.
\end{equation}
with a suitable choice of the initial-terminal conditions (by a Riemann-Hilbert approach).
By taking $H(x,p)=\frac{1}{2}p^2$ and $f(x,m)=m$ (a local coupling), and letting 
\[
h_\nu(t,x):=-u_{\nu/2}(T-t,x),\quad \rho_\nu(t,x):=-m_{\nu/2}(T-t,x),
\]
the equations \eqref{1.1} specify to \eqref{1.2}.
Thus, our result on the convergence of second-order MFGs to the vanishing viscosity limit 
stipulates how the lower-tail limit of the most probable KPZ path in ($1+1$)-dimension is obtained,
and gives a quantitative rate in the large deviation limit thereof.
Of course, the Cole-Hopf transform from the SHE to the KPZ equation and the weak noise theory is only valid in dimension $d = 1$.
There is no obvious theory for dimension $d \ge 2$ (see \cite{CSZ17, CD20, MU18} for recent development),
so it is not clear how large deviations of the KPZ equation in dimension $(d+1)$ with $d \ge 2$ is connected to MFGs.
Nevertheless, our results for MFGs hold for general dimensions. 
\end{enumerate}

\quad Now we turn back to MFGs.
The well-posedness of \eqref{1.1} has been thoroughly studied in the case of nonlocal, monotone, and regularizing couplings
for both second-order MFGs ($\nu>0$) and first-order MFGs ($\nu=0$) \cite{LL06a, LL06b, LL07},
and for time-dependent MFGs \cite{ferreira2021existence}.
When the function $f(x,m)$ depends locally on the value of $m$ and $\nu>0$, 
the well-posedness problem has been addressed by \cite{GPS2, GPS1} for classical solutions,
and by \cite{por14, POR} for weak solutions. 
When $\nu=0$, in general, one cannot expect the existence of classical solutions. 
 \cite{car} exploited the variational method, and showed that first-order MFGs can be viewed as an optimality condition for two convex problems. This approach can also be used for quadratic Hamiltonian MFGs in the whole domain \cite{orrieri}.
Recently, allowing the terminal condition $\bar u$ to depend on the density $m_\nu(T,\cdot)$, 
\cite{classical} obtained the weak solution as the limit of a sequence of classical solutions to strictly elliptic problems. 

\quad Here is an overview of our main results {on the convergence rate of vanishing viscosity approximations for MFGs.}
For both local and nonlocal couplings,
it is known that as $\nu \to 0^+$, 
the solutions $(u_\nu,m_\nu)$ converge to $(u,m)$, 
where $(u,m)$ are solutions to first-order MFGs \cite{achdou2021mean,note,CGPT}. 
However, the convergence rate is not well understood.
This paper provides the first quantitative rate of vanishing viscosity approximations for MFGs.
\begin{enumerate}[itemsep = 3 pt]
\item[(a)]
{\bf Local coupling}. 
When the coupling $f$ is local, 
we apply a duality approach which relies on the fact that 
equation \eqref{1.1} is the optimality condition for two convex optimization problems.
This approach was used in \cite{S17}, and can be traced back to \cite{bre99}. 
To obtain a convergence rate of solutions as $\nu \to 0^+$, the main assumption is the coercivity of the Hamiltonian and the coupling, 
as specified in the conditions (H5-1) and (H5-2).
Similar assumptions were made in \cite{GM, S17} to get a Sobolev estimate for $m$.
With these assumptions, we prove that $(u_\nu, m_\nu)$ converges in some Sobolev norm at a polynomial rate as $\nu\to 0^+$ (see Theorem \ref{T.4.3}).  

To illustrate, take $H(x,p) = |p|^r$ and $f(x,m) = m^{q-1}$ for some $q,r > 1$
(while our results hold for more general Hamiltonian and local couplings). 
Let $$\beta:= \max\left\{\frac{qr}{2qr - q -r}, 1\right\}(d+1) - d \ge 1.$$
Our result shows that
\begin{subequations}
\begin{align}
& \iint_\Omega (m_\nu^{q/2} - m^{q/2})^2 dx dt \lesssim \nu^{\frac{1}{1+\beta}}, \\
\label{eq:qrrate}
& \iint_\Omega \left||Du_\nu|^{r/2-1}Du_\nu - |Du|^{r/2-1}Du\right|^2 m dx dt \lesssim \nu^{\frac{1}{1+\beta}}.
\end{align}
\end{subequations}
{In particular, we have $\iint_\Omega (m_\nu - m)^2 dx dt \lesssim \nu^{\frac{2}{q(1 + \beta)}}$ for $q \ge 2$ (see Remark \ref{rk.4.2}).}
Note that if $\frac{1}{q} + \frac{1}{r} \le 1$, then $\beta = 1$ and the rate in \eqref{eq:qrrate} is $\nu^{\frac{1}{2}}$ which is independent of the dimension $d$;
if $\frac{1}{q} + \frac{1}{r} > 1$, then the rate becomes $\nu^{\frac{c}{d}}$ for some $c > 0$ 
which decays slowly as the dimension $d$ is large.
So there is no curse of dimensionality if the growth of $H$ and $f$ is sufficiently large (i.e. $\frac{1}{q} + \frac{1}{r} \le 1$), while
the convergence may suffer from the curse of dimensionality otherwise.
However, we do not know whether the rate $\nu^{\frac{1}{1 + \beta}}$ is tight;
if it is, or the rate is $\nu^{\kappa(d)}$ with any $\kappa$ decreasing to $0$ as $d \to \infty$ for $\frac{1}{q} + \frac{1}{r} > 1$,
then it implies a phase transition in the dimension dependence of vanishing viscosity approximations for MFGs 
at $\frac{1}{q} + \frac{1}{r} = 1$.

When the Hamiltonian is quadratic ($q = 2$), 
we prove a stronger convergence result for $u_\nu$:
\begin{equation}
\label{eq:qrrate2}
\iint_\Omega |u_\nu - u|^2 m dx dt \lesssim \nu^{\frac{1}{2(1+\beta)}} \quad \mbox{if } r > \max \left\{ 2 +\frac{d}{d+1}, \frac{d^2}{2d +3} \right\},
\end{equation}
(see Theorem \ref{T.5.3}). This result, presented in Theorem \ref{T.5.3}, cannot be directly deduced from \eqref{eq:qrrate}, as the weighted Poincar\'e inequality may not be applicable. 
Instead, we use the higher regularity of $m_\nu$ (see \cite{GM}) and the equations. 
For $d = 1$, the condition in \eqref{eq:qrrate2} requires $r > \frac{5}{2}$, which fails to cover the KPZ Hamiltonian with $r = 2$. 
To address this, we further need $d \le 3$ to get uniform boundedness of $u_\nu$ in the KPZ setting with $q = r = 2$, as presented in Theorem \ref{T.5.4}.

We also mention that if the terminal data $\bar u=\bar u(x,m_\nu(T,x))$ depends on $m_\nu$, 
the convergence rate of vanishing viscosity for MFGs remains open.
In this case, we have no longer the optimality condition characterization.

\smallskip
\item[(b)]
{\bf Nonlocal and regularizing coupling}. 
When the coupling is nonlocal and regularizing, we adopt a different approach. 
Instead of using the optimization structure, we rely on the duality of the two equations and some regularity properties of the solutions, both uniform and non-uniform with respect to $\nu$. 
Our main tool is the uniform semi-concavity of the solution $u_\nu$, which results from the superlinear growth of the Hamiltonian and classical Hamilton-Jacobi equations \cite{semiconcave}. 
This property allows us to establish a uniform bound on the $W^{1,2}$ norm of $\nu^{\frac12} m_\nu$ for all $\nu > 0$. 
With these findings, we first prove \eqref{eq:qrrate} with $r=2$ and $\beta=1$, and 
\beq\lb{222}
\iint_\Omega (f(x,m_\nu)-f(x,m))(m_\nu-m)dxdt\lesssim \nu^{\frac12},
\eeq
(see Theorem \ref{T.5.2}). This method, while requiring less restrictive assumptions, leads to a faster convergence rate 
compared to those with a local coupling.

Assuming that \eqref{222} implies pointwise convergence of $f(x,m_\nu)$ to $f(x,m)$  as $\nu\to 0^+$, 
we obtain a pointwise convergence of $u_\nu$ with rate $\nu^\frac14$ (see Theorem \ref{T.6.3'}).
While under a weaker condition (H4''), the estimate \eqref{222} implies that $f(x,m_\nu)$ converges to $f(x,m)$ in $L^1(\Omega)$ with a rate $\nu^{1/4}$. 
Then we use both the dual equation method (see \cite{lin}) and the viscosity solution method (see \cite{two}) to prove that the HJ equation of $u$ is stable under both $L^1$-perturbation of coefficients and vanishing viscosity. Indeed, we obtain for all $t \le T$,
\[
\label{eq:unonlocal}
\|u_\nu(t, \cdot) - u(t, \cdot)\|_{L^1(\Omega)} \lesssim \nu^{\frac{1}{4}},
\]
(see Theorem \ref{T.6.3}).
\end{enumerate}
We also mention that in the literature, $\bar m$ is often assumed to be strictly positive, see \cite{CG, CGPT, classical, porretta2023regularizing} for superlinear Hamiltonian, and \cite{gra} for linear Hamiltonian. 
The condition is removed under extra requirements on $H$ and $f$ \cite{car, GM19}, or for the kinetic type first order MFGs (with continuous initial measure) \cite{griffin2022variational}.
Relying on \cite{CG, CGPT},
we show that this condition can be dropped without further restrictions (see Theorem \ref{T.3.2}).

\medskip
\quad The remainder of the paper is organized as follows.
In Section \ref{S2}, we provide background on MFGs with a local coupling. 
In Section \ref{S3}, we prove the well-posedness of MFGs with nonnegative data.
In Sections \ref{S4} and \ref{S42}, we study the convergence rate of vanishing viscosity for MFGs with a local coupling.
In Sections \ref{S5} and \ref{S61} we consider the convergence rate of vanishing viscosity for MFGs with a nonlocal and regularizing coupling.
Finally, a list of open problems are presented in Section \ref{S8}.

\section{Assumptions and Preliminaries for Local Coupling}\lb{S2}

\quad We first discuss the assumptions for the case of local coupling $f$. The assumptions are made so that \eqref{1.1} is well-posedness for all $\nu\geq 0$, and most of them can be found in \cite{CG,CGPT}. We assume that there exists $C_0\geq 1$ such that:

\begin{itemize}
    \item[(H1)] (Conditions on the coupling) $f:\bbT^d\times [0,\infty)\to \bbR$ is continuous in both variables, strictly increasing with respect to the second variable, and there exists $q>1$ such that
    \[
    C_0^{-1}m^{q-1}-C_0\leq f(x,m)\leq C_0m^{q-1}+C_0\quad\text{     for all $m\geq 0$ and $x\in\bbT^d$}. 
    \]
    Moreover, we make the normalization condition:
    \beq\lb{H1f}
    f(x,0)=0\quad\text{ for all }x\in\bbT^d.
    \eeq

    \item[(H2)] (Conditions on the Hamiltonian) The Hamiltonian $H:\bbT^d\times\bbR^d\to\bbR$ is continuous in both variables, strictly convex and differentiable in the second variable, with $D_pH$ continuous, and satisfying for some $r>1$,
    \[
    C_0^{-1}|p|^r-C_0\leq H(x,p)\leq C_0|p|^r+C_0 \quad\text{ for all }  (x,p)\in \bbT^d\times\bbR^d.
    \]

    \item[(H3)] (Conditions on the initial and terminal data)  $\bar u:\bbT^d\to\bbR$ is of class $\calC^2$, and $\bar{m}:\bbT^d\to \bbR$  is a $\calC^1$ nonnegative density function.

\end{itemize}

\begin{remark}
1. As discussed in \cite{CGPT}, \eqref{H1f} is just a normalization condition, which can be assumed without loss of generality. In fact, if it does not hold, one can replace $f(x,m)$ and $H(x,p)$ by $f(x,m)-f(x,0)$ and $H(x,p)-f(x,0)$, respectively.

\smallskip

\quad 2. Recall that the Fenchel conjugate $H^*(x,\cdot)$ of $H(x,\cdot)$ for each $x\in\bbT^d$ is defined as $H^*(x,\xi):=\sup_{p\in\bbR^d}(\langle \xi,p\rangle-H(x,p))$. Then $H^*$ is continuous and satisfies for some $C_0\geq1$ (without loss of generality let us still use $C_0$) such that
\[
C_0^{-1}|\xi|^{r'}-C_0\leq H^*(x,\xi)\leq C_0|\xi|^{r'}+C_0\quad \text{ for all }(x,\xi)\in\bbT^d\times\bbR^d,
\]
where $r':=\frac{r}{r-1}$ is the conjugate of $r$.
 Later we also write $q':=\frac{q}{q-1}$ as the conjugate of $q$.

\smallskip
\quad 3. Let $F$ be defined as
\[
F(x,m):=\int_0^mf(x,s)ds\quad\text{ if }m\geq 0,
\]
and $F(x,m)=+\infty$ if $m<0$. Then {\rm(H2)} yields for some $C_0\geq 1$,
 \beq\lb{F}
    C_0^{-1}m^{q}-C_0\leq F(x,m)\leq C_0m^{q}+C_0\quad\text{     for all $m\geq 0$ and $x\in\bbT^d$}. 
\eeq
We define $F^*(x,\cdot)$ to be the Fenchel conjugate of $F(x,\cdot)$. Then $F^*(x,\alpha)$ is strictly convex in $\alpha$, $F^*(x,\alpha)=0$ for $\alpha\leq 0$, and for some $C_0\geq 1$,
 \beq\lb{F*}
    C_0^{-1}\alpha^{q'}-C_0\leq F^*(x,\alpha)\leq C_0 \alpha^{q'}+C_0\quad\text{     for all $\alpha\geq 0$ and $x\in\bbT^d$}. 
\eeq

\quad 4. Unlike \cite{CG,CGPT}, we do not need to assume $\bar m>0$. It was used in \cite{CG,CGPT} to show the existence of a solution for the optimization problem \eqref{opt2}. 

\smallskip
\quad 5. One can consider equations with more general second order terms:
\[
\left\{
\begin{aligned}
&-\partial_t u_\nu-\nu A_{ij}\partial_{ij}u_\nu+H(x,Du_\nu)=f(x,m_\nu),\\
&\partial_t m_\nu-\nu\partial_{ij}(A_{ij} m_\nu)-\nabla\cdot(m_\nu D_pH(x,Du_\nu))=0,\\
&m_\nu(0,x)=\bar m(x),\quad u_\nu(T,x)=\bar u(x)
\end{aligned}
\right.
\]
where $A=A(x)$ is assumed to be a Lipschitz continuous map, taking values in the set of symmetric, uniformly positive definite matrices. But for this, one needs to further assume $r\geq q'$, see \cite{CGPT}.

\end{remark}

\subsection{Optimization problems}
We discuss two optimal control problems which are in duality, and we refer to \cite{car,CGPT}. 
For any $\nu\geq 0$, the first problem is
\beq\lb{opt1}
\inf_{(m,w)\in\calK_{1,\nu}}\calB(m,w)
\eeq
where
\[
\calB(m,w):=\iint_{\Omega}mH^*\left(x,-\frac{w}{m}\right)+F(x,m)dxdt+\int_{\bbT^d}\bar{u}(x)m(T,x)dx,
\]
with $\Omega=(0,T)\times\bbT^d$, and
\[
\calK_{1,\nu}:=\left\{(m,w)\in L^1(\bbT^d)\times L^1(\bbT^d;\bbR^d)\,\big|\, \partial_t m-\nu\Delta m+\nabla\cdot w=0,\, m(0)=\bar{m}\right\}
\]
where the continuity equation holds in the sense of distributions.
When $m=0$, we use the usual convention:
\[
mH^*\left(x,-\frac{w}{m}\right):=\left\{
\begin{aligned}
    &+\infty, &\text{ if $m=0$ and $w\neq 0$},\\
    &0,  &\text{ if $m=0$ and $w= 0$}.
\end{aligned}\right.
\]


\quad Now we discuss the second minimization problem. 
Recall $q,r>1$ from (H1)(H2) and $q',r'$ are their conjugates, respectively. Set
\beq\lb{gamma}
\gamma:=\frac{rq'(d+1)}{d-r(q'-1)}\,\text{ if }q'<1+\frac{d}{r}\quad\text{and}\quad\gamma:=\infty\,\text{ if }q'> 1+\frac{d}{r},
\eeq
and $\gamma>0$ can be an arbitrarily large constant when $q'=1+\frac{d}{r}$.
Then we let 
\[
\calA(u,\alpha):=\iint_{\Omega} F^*\left(x,\alpha(t,x)\right)dxdt-\int_{\bbT^d}{u}(0,x)\bar m(x)dx,
\]
and let $\calK_{2,\nu}$ be the set of $(u,\alpha)\in L^\gamma(\Omega)\times L^{q'}(\Omega)$ such that $Du\in L^r(\Omega)$ and the following holds in the sense of distributions
\[
-\partial_tu-\nu\Delta u+H(x,D u)\leq \alpha,\quad u(T,\cdot)\leq \bar{u}.
\]
The precise meaning of the inequality is given in \cite[Section 3]{CGPT}.
The second optimization problem (also called the relaxed problem in \cite{car,CGPT}) is 
\beq\lb{opt2}
\inf_{(u,\alpha)\in\calK_{2,\nu}}\calA(u,\alpha).
\eeq

\quad It turns out that for each $\nu\geq 0$, the optimization problems \eqref{opt1} and \eqref{opt2} are in duality. The first equality below is proved in \cite{car,CGPT} by the Fenchel-Rockafellar theorem.
The second equality is due to the fact that one can always replace $\alpha$ by $\max\{\alpha,0\}$ as $F^*(x,\alpha)=0$ for $\alpha\leq 0$. 
Here we remark that $\bar m>0$ is not needed in the proofs.

\begin{theorem}[{\rm }\cite{CGPT}]\lb{T.2.1} For all $\nu\geq 0$,
\[
-\min_{(m,w)\in \calK_{1,\nu}}\calB(m,w)
=\inf_{(u,\alpha)\in\calK_{2,\nu}}\calA(u,\alpha)=\inf_{(u,\alpha)\in\calK_{2,\nu},\,\alpha\geq 0\text{ a.e.}}\calA(u,\alpha).
\]
Moreover, the minimum of the first term is achieved by a unique pair in $(m,w)\in\calK_{1,\nu}$ satisfying $(m,w)\in L^q(\Omega )\times L^{\frac{r'q}{r'+q-1}}(\Omega)$.
\end{theorem}

\quad The last statement that the minimum is achieved by a unique pair in $\calK_{1,\nu}$ is because
the set $\calK_{1,\nu}$ is convex, and the functions $F(x,\cdot)$ and $H^*(x,\cdot)$ are strictly convex for each $x$.

\section{Well-posedness for non-negative data}
\lb{S3}

\quad In this section, we show well-posedness of \eqref{1.1} without assuming $\bar m$ to be strictly positive.
First we recall the notion of weak solutions in \cite{CGPT}. Let $\nu\geq 0$ and $\gamma$ be from \eqref{gamma}.

\begin{definition}
We say that $(u_\nu,m_\nu)\in L^\gamma(\Omega)\times L^q(\Omega)$ is a weak solution to \eqref{1.1} if
\begin{itemize}
    \item[(i)] the following integrability conditions hold:
    \[
    Du_\nu\in L^r,\, m_\nu H^*(\cdot,D_pH(\cdot,Du_\nu))\in L^1\quad\text{and}\quad m_\nu D_pH(\cdot,Du_\nu)\in L^1.
    \]
    \item[(ii)] The following  (in)equalities hold in the sense of distribution:
    \[
    -\partial_t u_\nu-\nu\Delta u_\nu+H(x,Du_\nu)\leq f(x,m_\nu)\quad\text{in }\Omega,
    \]
    with $u_\nu(T,\cdot)\leq \bar u$, and
    \[
    \partial_t m_\nu-\nu\Delta m_\nu-\nabla\cdot(m_\nu D_p(x,Du_\nu))=0\quad\text{in }\Omega\text{ with }m_\nu(0)=\bar{m}.
    \]

    \item[(iii)] We have
    \[
    \iint_\Omega m_\nu f(x, m_\nu)+H^*(x,D_pH(x,Du_\nu))dxdt+\int_{\bbT^d}m_\nu(T,x)\bar u(x)-\bar{m}(x)u_\nu(0,x)dx=0.
    \]
Here the last term is well-defined due to Lemma 5.1 \cite{CGPT}.
\end{itemize}

\end{definition}

\begin{theorem}\lb{T.3.2}
Assume {\rm (H1)--(H3)} and let $\nu\geq 0$. The problem \eqref{opt2} has at least one solution in $\calK_{2,\nu}$. Let $(m_\nu,w_\nu)\in\calK_{1,\nu}$ be the minimizer of \eqref{opt1} and $(u_\nu,\alpha_\nu)\in \calK_{2,\nu}$ be a minimizer of \eqref{opt2}. Then:
\begin{enumerate}[itemsep = 3 pt]
\item
$(u_\nu,m_\nu)$ is  a weak solution to \eqref{1.1}, and $w_\nu=-m_\nu D_pH(\cdot,Du_\nu)$ and $\alpha_\nu=f(\cdot,m_\nu)$ almost everywhere. 
\item
Conversely, any weak solution $(u_\nu,m_\nu)$ of \eqref{1.1} is such that the pair $(m_\nu,-m_\nu D_pH(\cdot,Du_\nu))$ is the minimizer of \eqref{opt1} and $(u_\nu,f(\cdot, m_\nu))$ is a minimizer of \eqref{opt2}.
\item
The solution is unique in the following sense: if $(u_\nu',m_\nu')$ is another weak solution to \eqref{1.1}, then $m_\nu=m_\nu'$ a.e. and $u_\nu=u_\nu'$ in $\{m_\nu>0\}$.
\end{enumerate}
\end{theorem}

%

\quad The theorem is proved in \cite{CGPT} with an extra assumption that $\bar{m}$ is strictly positive. 
The assumption is only used in their Proposition 5.4 to show the existence of a solution to the optimization problem \eqref{opt2}. 
We use a uniform lower bound of $u_\nu$ and a slightly different test function to remove the constraint. We state the result below.

\begin{proposition}\lb{P.3.3}
Under the assumptions {\rm (H1)--(H3)}, for each $\nu\geq 0$, the problem \eqref{opt2} has at least one solution $(u_\nu,\alpha_\nu)\in \calK_{2,\nu}$ which is uniformly bounded from below by a constant depending only on the assumptions for all $\nu\in[0,1]$.

\quad Moreover, there exists $C>0$ depending only on the assumptions such that
\beq\lb{3.4}
\sup_{\nu\in[0,1]}\|u_\nu(0,\cdot)\|_{L^1(\bbT^d)}+\|Du_\nu\|_{L^r(\Omega)}+\|\alpha_\nu\|_{L^{q'}(\Omega)}\leq C.
\eeq
\end{proposition}

\begin{remark}
Let us comment that here $\|u_\nu(0,\cdot)\|_{L^1(\bbT^d)}$ is well-defined due to Lemma 5.1 \cite{CGPT} ($u_\nu$ has a ``trace'' in a weak sense).
\end{remark} 

\begin{proof}
The proof follows largely the one of Proposition 5.4 \cite{CGPT}. 

\quad {\bf Step 1}. From the first paragraph of their argument, there exists a minimizing sequence $(u_\nu^n,\alpha_\nu^n)\in \calK_{2,\nu}$ for problem \eqref{opt2} such that $u_\nu^n,\alpha_\nu^n$ are continuous, $\alpha^n_\nu\geq 0$, and $u_\nu^n$ is a viscosity solution to
\beq\lb{3.1}
-\partial_tu_\nu^n-\nu\Delta u_\nu^n+H(x,D u_\nu^n)= \alpha_\nu^n,\quad u_\nu^n(T,\cdot)= \bar{u}.
\eeq
Since $H$ is convex, this equality also holds in the sense of distribution \cite{16}.
Moreover, since $\alpha^n_\nu\geq 0$, comparing the solution with $-\|\bar u\|_\infty-\|H(\cdot,0)\|_\infty (T-t)$ yields
\beq\lb{3.2}
u_\nu^n(t,x)\geq -\|\bar u\|_\infty-\|H(\cdot,0)\|_\infty (T-t)\quad \text{ in }\Omega.
\eeq
It is clear that the right-hand side is independent of $\nu$.

\quad {\bf Step 2}. Next, we show some bounds for $\alpha_\nu^n$ and $u_\nu^n$ that are uniform in $n$ and $\nu$. We integrate \eqref{3.1} against $\bar m+1$ on $\Omega$ (instead of $\bar m$ that was used in \cite{CGPT}) to get
\begin{align*}
 \int_{\bbT^d}u_\nu^n(0,x)(\bar{m}(x)+1)dx+&\iint_{\Omega}\nu D\bar{m}Du_\nu^n+(\bar{m}+1)H(x,Du_\nu^n)dxdt\\
&=\iint_{\Omega}(\bar{m}(x)+1)\alpha_\nu^n(t,x)dxdt+\int_{\bbT^d}\bar{u}(\bar{m}+1)dx.  
\end{align*}
By (H2) and (H3) ($\bar m,\bar u\in\calC^1$), there exists $C\geq 1$ depending only on (H3) and $C_0$ such that
\beq\lb{3.5}
\begin{aligned}
\int_{\bbT^d}u_\nu^n(0,x)dx-&C\nu\iint_{\Omega}|Du_\nu^n|dxdt+ C^{-1}\iint_{\Omega}|Du_\nu^n|^rdxdt\\
&\qquad\qquad \leq -\int_{\bbT^d}u_\nu^n(0,x)\bar{m}(x)dx+C\iint_{\Omega}|\alpha_\nu^n|dxdt+C.    
\end{aligned}
\eeq
Then, we use \eqref{3.2} to get
\[
\int_{\bbT^d}u_\nu^n(0,x)dx\geq \int_{\bbT^d}|u_\nu^n(0,x)|dx-C
\]
for some $C$ depending only on $\|\bar u\|_\infty,\|H(\cdot,0)\|_\infty$ and $T$.
Next, using that $r> 1$ and $q'>1$, we get for any $\eps>0$ there is $C_\eps$ satisfying
\[
\iint_{\Omega}|Du_\nu^n|dxdt\leq \eps \iint_{\Omega}|Du_\nu^n|^rdxdt+C_\eps,
\]
\[
\iint_{\Omega}|\alpha_\nu^n|dxdt\leq \eps\iint_{\Omega}|\alpha_\nu^n|^{q'}dxdt+C_\eps.
\]
Applying these in \eqref{3.5} and taking $\eps$ to be sufficiently small, we obtain for some $C$ independent of $n$ and $\nu\in[0,1]$,
\begin{align*}
\int_{\bbT^d}|u_\nu^n(0,x)|dx& +(2C)^{-1}\iint_{\Omega}|Du_\nu^n|^rdxdt+(2C_0)^{-1}\iint_{\Omega}|\alpha_\nu^n|^{q'}dxdt\\
&\leq -\int_{\bbT^d}u_\nu^n(0,x)\bar{m}(x)dx+C_0^{-1}\iint_{\Omega}|\alpha_\nu^n|^{q'}dxdt+C      
\end{align*}
where $C_0\geq 1$ is from \eqref{F*}. It follows from \eqref{F*} and the definition of $\calA$ that
\begin{align*}
&\int_{\bbT^d}|u_\nu^n(0,x)|dx+(2C)^{-1}\iint_{\Omega}|Du_\nu^n|^rdxdt+(2C_0)^{-1}\iint_{\Omega}|\alpha_\nu^n|^{q'}dxdt\\
&\qquad\qquad \leq -\int_{\bbT^d}u_\nu^n(0,x)\bar{m}(x)dx+C_0^{-1}\iint_{\Omega}F^*(x,\alpha_\nu^n)dxdt+C_0+C \leq \calA(u^n_\nu,\alpha_\nu^n)+C'.    
\end{align*}

\quad {\bf Step 3}. Since $(u^n_\nu,\alpha_\nu^n)$ is a minimizing sequence, for all $n$ sufficiently large and $\bar\alpha_\nu:=-\partial_t\bar u-\nu\Delta \bar u+H(x,D \bar u)$ we obtain
\[
\calA(u^n_\nu,\alpha^n_\nu)\leq \sup_{\nu\in[0,1]}\inf_{(u,\alpha)\in\calK_{2,\nu}}\calA(u,\alpha)+1\leq \sup_{\nu\in[0,1]}\calA(\bar{u},\bar{\alpha}_\nu)+1,
\]
which is finite. So there exists $C>0$ independent of $n$ and $\nu\in[0,1]$ such that for all large $n$,
\beq\lb{3.3}
\int_{\bbT^d}|u_\nu^n(0,x)|dx+\iint_{\Omega}|Du_\nu^n|^rdxdt+\iint_{\Omega}|\alpha_\nu^n|^{q'}dxdt\leq C.
\eeq
With the uniform bound, one can take weak limit of $(u^n_\nu,\alpha^n_\nu)$ along a subsequence of $n\to\infty$. Following Step 2 and Step 3 in Proposition 5.4 \cite{CGPT}, we obtain that the weak limit $(u_\nu,\alpha_\nu)$ is a minimizer of the optimization problem \eqref{opt2}. Finally, \eqref{3.4} follows from \eqref{3.3}.
\end{proof}

\quad The qualitative vanishing viscosity limit is studied in \cite[Theorem 6.5]{CGPT}. We state it below.

\begin{theorem}[{\rm }\cite{CGPT}]\lb{T.2.2}
Let $(u_\nu,m_\nu)$ be a weak solution to \eqref{1.1} with $\nu> 0$. Then $m_\nu$ converges strongly to $m$ in $L^q(\Omega)$ as $\nu\to 0$, and $u_\nu$ converges weakly up to a subsequence to $u$ in $L^\gamma(\Omega)$, where the pair $(u,m)$ is a weak solution to \eqref{1.1} with $\nu=0$.
\end{theorem}

\section{Vanishing Viscosity Limit -- Local Coupling}
\lb{S4}

\quad To study the convergence rate of vanishing viscosity limit of mean field games, we further need the following regularity and coercivity assumptions.

\begin{itemize}
\item[(H4)] (Regularity condition) There exist $C_0,C_1,C_2\geq 0$ such that for all $x,y\in\bbT^d$ and $m\geq 0$,
\[
|F(x+y,m)-F(x,m)|\leq C_0(1+m^{q})|y|,
\]
and for any $p\in\bbR^d$ and $\sigma\in(0,1]$,
\beq\lb{4.0}
\sup_{z\in B_\sigma} H\left(x+z,\frac{p}{1+C_1\sigma}\right)\leq \frac{H(x,p)}{1+C_1\sigma}+C_2\sigma.
\eeq

\smallskip

\item[(H5-1)] (Coecivity condition on $H$) There are $J_1,J_1^*:\bbT^d\times\bbR^d\to\bbR^d$ and $c_0>0$ such that for all $p,\xi\in\bbR^d$ we have
\beq\lb{H51}
H(x,p)+H^*\left(x,\xi\right)-p\cdot \xi\geq  c_0\,|J_1(x,p)-J_1^*(x,\xi)|^2.
\eeq

\smallskip

\item[(H5-2)] (Coecivity condition on $F$) There are $J_2,J_2^*:\bbT^d\times\bbR\to\bbR$ and $c_0>0$ such that for all $m,\alpha\in\bbR$ we have
\beq\lb{H52}
F(x,m)+F^*(x,\alpha)-m\alpha\geq c_0\,|J_2(x,m)-J_2^*(x,\alpha)|^2.
\eeq
\end{itemize}
\begin{remark}\lb{R.4.1}
1. If $H(x,p)$ is only a function of $p$, then we can take $C_1=C_2=0$ in (H4) and \eqref{4.0} is trivial. 

\quad The condition \eqref{4.0} on $H$ is stronger than the one on $F$, but it is still satisfied by a large class of Hamiltonian. For example, if $H(x,p)=h_1(x)|p|^r+h_2(x)$ with two Lipschitz continuous functions $h_1,h_2:\bbT^d\to\bbR$ and $h_1\geq 0$, then \eqref{4.0} holds with $C_1:=\|Dh_1\|_{\infty}/(r-1)$ and $C_2:=\|Dh_2\|_{\infty}$.

\smallskip

\quad 2. Note that $F^*(x,\alpha)=0$ for $\alpha\leq 0$. For $\alpha\geq 0$, there exists a unique $m_\alpha\geq 0$ such that
\[
F^*(x,\alpha)=\sup_{m\geq 0}\{\alpha m-F(x,m)\}=\alpha m_\alpha-F(x,m_\alpha).
\]
By \eqref{F}, $m_\alpha\geq0$ satisfies
\beq\lb{F*1'}
 m_\alpha^q\leq C\alpha^{q'}+C\quad\text{ for some $C>0$}.
\eeq
So it follows from {\rm(H4)} that for all $x,y\in\bbT^d$
and $\alpha\geq 0$,
\beq\lb{F*1}
|F^*(x+y,\alpha)-F^*(x,\alpha)|\leq C(1+\alpha^{q'})|y|\quad\text{for some $C>0$.}
\eeq

\quad It is not hard to see that $F^*$ is increasing in $\alpha$. Indeed for any $\sigma\geq 0$, 
\begin{align}\lb{F*2}
F^*(x,\alpha+\sigma)\geq (\alpha+\sigma)m_\alpha-F(x,m_\alpha)    \geq F^*(x,\alpha).
\end{align}
Moreover, by \eqref{F*1'}, for any $\sigma\geq 0$,
\beq\lb{F*3}
\begin{aligned}
F^*(x,\alpha)
\geq F^*(x,\alpha+\sigma)-\sigma m_{\alpha+\sigma}\geq F^*(x,\alpha+\sigma)-C\sigma(1+(\alpha+\sigma)^{q'/q}),
\end{aligned}
\eeq
where $m_{\alpha+\sigma}$ is such that $F^*(x,\alpha+\sigma)=\alpha m_{\alpha+\sigma}-F(x,m_{\alpha+\sigma})$

\smallskip

\quad 3. 
Let $\tau:\bbT^d\to (0,\infty)$ and $h:\bbT^d\to\bbR$ be continuous. If $H(x,p)=\frac{1}{r}|\tau(x)p|^r+h(x)$ for some $r>1$, then $H^*(x,\xi)=\frac{1}{r'}|\xi/\tau(x)|^{r'}-h(x)$. In this case, \eqref{H51} holds with 
\[
J_1(x,p)=\tau(x)^{r/2}|p|^{r/2-1}p, \quad J_1^*(x,\xi)=\tau(x)^{-r'/2}|\xi|^{r'/2-1}\xi\quad\text{and}\quad c_0= (\max\{r,r'\})^{-1}.
\]
If $F(x,m)=\frac{1}{q}(\tau(x)m)^q$ for some $q>1$, then $F^*(x,\alpha)=\frac{1}{q'}(\alpha/\tau(x))^{q'}$, and \eqref{H52} holds with
\[
J_2(x,m)=\tau(x)^{q/2}m^{q/2},\quad J_2^*(x,\alpha)=\tau(x)^{-q'/2}\alpha^{q'/2}\quad\text{and}\quad c_0= (\max\{q,q'\})^{-1}.
\]
We only prove \eqref{H52} with this selection of functions in the appendix. The proof for \eqref{H51} is almost the same.

\quad Finally, let us comment that in the above examples the value of $c_0$ is optimal. Indeed, if $\tau\equiv 1,q=q'=2$, it is easy to see that $c_0\leq \frac12$. The optimality can also be seen in the proof.
\end{remark}

\quad The following lemma is a consequence of the coercivity condition. The proof is identical to the one of Lemma 3.2 \cite{S17}, and we skip it.

\begin{lemma}\lb{L.4.1}
Assume {\rm(H1)--(H4)}, and for $\nu\geq 0$, let $(u,\alpha)\in \calK_{2,\nu}$ and $(m,w)\in \calK_{1,\nu}$. If {\rm(H5-1)} holds, we have 
\beq\lb{4.6}
\begin{aligned}
\calA(u,\alpha)+\calB(m,w)
\geq c_0\|m^{1/2}(J_1(\cdot,Du)-J_1^*(\cdot,-w/m))\|^2_{L^2(\Omega)}.
\end{aligned}
\eeq
And if {\rm(H5-2)} holds, we have 
\beq\lb{4.6'}
\begin{aligned}
\calA(u,\alpha)+\calB(m,w)\geq c_0\left\|J_2(\cdot,m)-J_2^*(\cdot,\alpha)\right\|^2_{L^2(\Omega)}.
\end{aligned}
\eeq
\end{lemma}

\quad For $\nu\geq 0$, let $(m_\nu,w_\nu)\in\calK_{1,\nu}$ be the minimizer of \eqref{opt1} and $(u_\nu,\alpha_\nu)\in \calK_{2,\nu}$ be a minimizer of \eqref{opt2}. Theorem \ref{T.3.2} and Theorem \ref{T.2.1} yield that $(u_\nu,m_\nu)$ is a weak solution to \eqref{1.1}, and
\beq\lb{4.9}
\calA(u_\nu,\alpha_\nu)+\calB(m_\nu,w_\nu)=0.
\eeq
The following result is a direct consequence of Lemma \ref{L.4.1} and \eqref{4.9}.

\begin{corollary}
We have for any $\nu\geq 0$ and for a.e. $(t,x)\in\Omega$ such that
\beq\lb{4.8}
J_1^*(x,-w_\nu/m_\nu)=J_1(x,Du_\nu)\quad\text{and}\quad J_2^*(x,\alpha_\nu)=J_2(x,m_\nu).
\eeq
\end{corollary}

\quad Now we prove the main theorem of the section.

\begin{theorem}\lb{T.4.3}
Assume {\rm(H1)--(H4)}, and let $(m_\nu,\alpha_\nu)$ be a weak solution to \eqref{1.1} with $\nu\geq 0$. 
Write $(u,\alpha,m,w)=(u_0,\alpha_0,m_0,w_0)$, and
\beq\lb{beta}
\beta:=\max\left\{\frac{q'r'}{q'+r'},1\right\}(d+1)-d\geq 1.
\eeq
Then if {\rm(H5-1)} holds, there exists $C>0$ such that for all $\nu\in[0,1]$ we have
\[
\iint_\Omega  |J_1(\cdot,Du_\nu) -J_1(\cdot,Du)|^2m\,dxdt \leq C\nu^{\frac{1}{1+\beta }},
\]    
and if {\rm(H5-2)} holds, for all $\nu\in[0,1]$ we have
\[
\iint_\Omega |J_2(\cdot,m_\nu)-J_2(\cdot,m)|^2dxdt \leq C\nu^{\frac{1}{1+\beta }}.
\]    
\end{theorem}

\begin{proof}
We use the optimization structure of \eqref{1.1} described in Theorem \ref{T.3.2}. 

\quad {\bf Step 1}. By Theorem \ref{T.2.1} and the proof of Proposition \ref{P.3.3}, we can take a minimizing sequence $(u^n,\alpha^n)\in\calK_{2,0}$ such that 
 $\alpha^n\geq 0$, and
\beq\lb{4.13'}
-\partial_tu^n+H(x,D u^n)= \alpha^n,\quad u^n(T,\cdot)= \bar{u}
\eeq
in the sense of distribution, and
\[
\calA(u^n,\alpha^n)\leq \calA(u,\alpha)+1/n.
\]
Due to \eqref{3.4}, there exists $C>0$ such that for all $n$ sufficiently large we have
\beq\lb{4.3}
\|Du^n\|_{L^r(\Omega)}+\|\alpha^n\|_{L^{q'}(\Omega)}\leq C.
\eeq

\quad {\bf Step 2}. The goal is to show that
\beq\lb{4.1}
\calA(u_\nu,\alpha_\nu)\leq \calA(u,\alpha)+C\nu^{\frac{1}{1+\beta }}.
\eeq
Below we construct $(u_\eps^n,\alpha_\eps^n)$ from $(u^n,\alpha^n)$ that can be used as a candidate for the optimization problem \eqref{opt2}.

\quad Let $\varphi:\bbR^d\to [0,\infty)$ be a mollifier i.e., $\varphi$ is smooth, compactly supported in $B_1$, and $\int_{\bbR^d}\varphi(x) dx=1$. Take $\eps:=\nu^{1/(1+\beta)}$ and set $\varphi_\eps:=\eps^{-d}\varphi(\cdot/\eps)$.
For $C_1,C_2$ from (H4), we define 
\beq\lb{4.10}
u^{n}_{\eps}:=\frac{\varphi_\eps*u^n}{1+C_1\eps}-(1+C_1)\eps\|\bar u\|_{\calC^1}+C_2\eps(t-T).
\eeq
It is direct to see from the definition that $u^{n}_{\eps}(T,\cdot)\leq \bar u$ as $\varphi_\eps$ is supported in $B_\eps$. 

\quad {\bf Step 3}. Because $H(x,\cdot)$ is convex, Jensen's inequality yields for each fixed $x_0\in\bbT^d$,
\[
H\left(x_0,D\left(\frac{\varphi_\eps*u^n}{1+C_1\eps}\right)(t,x)\right)\leq \varphi_\eps * H\left(x_0,D \left(\frac{u^n}{1+C_1\eps}\right)\right)(t,x).
\]
By \eqref{4.0} with $\eps$ in place of $\sigma$, we obtain
\[
H(x,Du^{n}_{\eps}(t,x))\leq \frac{\varphi_\eps * H(x,D u^n)(t,x)}{1+C_1\eps}+C_2\eps.
\]
We note that if $H$ is independent of $x$, then this holds with $C_1=C_2=0$.
We get from \eqref{4.13'} that
\beq\lb{4.15'}
-\partial_tu^{n}_{\eps}-\nu\Delta u^{n}_{\eps}+H(x,Du^{n}_{\eps})\leq \frac{\varphi_\eps*\alpha^n}{1+C_1\eps}-\frac{\nu}{1+C_1\eps}\Delta\varphi_\eps *u^n=:\alpha_\eps^{n}.
\eeq
In particular, we have $(u^{n}_{\eps},\alpha^{n}_\eps)\in\calA_{2,\nu}$. Consequently, as $(u_\nu,\alpha_\nu)$ is a minimizer of \eqref{opt2} over $\calK_{2,\nu}$, we obtain
\[
\calA(u_\nu,\alpha_\nu)\leq \calA(u^n_\eps,\alpha^n_\eps).
\]
It follows from the definition of $\eps$ that $\nu\eps^{-\beta}\leq 1$. We claim that 
\[
\calA(u_\nu,\alpha_\nu)\leq \calA(u^{n}_{\eps},\alpha_\eps^{n})\leq \calA(u^n,\alpha^n)+C\nu\eps^{-\beta}+C\eps.
\]
We postpone the proof to Lemma \ref{L.4.4} below. As a consequence of the claim, passing $n\to\infty$ yields \eqref{4.1}.

\quad {\bf Step 4}. Now, we combine \eqref{4.1} with \eqref{4.9} to get
\[
\calA(u_\nu,\alpha_\nu)+\calB(m,w)\leq \calA(u,\alpha)+\calB(m,w) +C\nu^{\frac{1}{1+\beta }}=C\nu^{\frac{1}{1+\beta}}.
\]
Then using \eqref{4.6} and (H5-1) yields
\begin{align*}
&c_0\,\|m^{1/2}(J_1(\cdot,Du_\nu) -J_1^*(\cdot,-w/m))\|^2_{L^2(\Omega)}\leq \calA(u_\nu,\alpha_\nu)+\calB(m,w)\leq C\nu^{\frac{1}{1+\beta}}.    
\end{align*}
Using instead \eqref{4.6'} and (H5-2) yields
\begin{align*}
&c_0\,\|J_2(\cdot,m)-J_2^*(\cdot,\alpha_\nu)\|^2_{L^2(\Omega)}\leq \calA(u_\nu,\alpha_\nu)+\calB(m,w)\leq C\nu^{\frac{1}{1+\beta}}.    
\end{align*}
Finally \eqref{4.8} yields the conclusion.
\end{proof}

\begin{remark} \label{rk.4.2}
    Suppose that $J_2(\cdot,m)\equiv m^{q/2}$. 
 We obtain from the second conclusion of theorem that when $q\geq 2$,
    \begin{align*}
\iint_\Omega|m_\nu-m|^2dxdt&\leq C\left(\iint_\Omega |m_\nu-m|^qdxdt\right)^{2/q}\\
&\leq C\left(\iint_\Omega |m_\nu^{q/2}-m^{q/2}|^2dxdt\right)^{2/q}\leq C\nu^{\frac{2}{q(1+\beta) }}    .        
    \end{align*}
\end{remark}

\begin{lemma}\lb{L.4.4}
Suppose $\nu\eps^{-\beta}\leq 1$. Then
\[
\calA(u^{n}_{\eps},\alpha_\eps^{n})\leq \calA(u^n,\alpha^n)+C\nu\eps^{-\beta}+C\eps.
\]
\end{lemma}
\begin{proof}
In view of \eqref{4.3}, Young's convolution inequality yields if $\frac{q'r'}{q'+r'}\geq 1$ (then $r\leq q'$),
\[
\|\Delta\varphi_\eps*u^n\|_{L^{q'}(\Omega)}\leq\|D\varphi_\eps\|_{L^{\beta_1}(\Omega)}\|Du^n\|_{L^r(\Omega)}\leq C\eps^{-{\beta_1}(d+1)+d}\quad\text{with }\beta_1:=\frac{q'r'}{q'+r'}\geq1,
\]
and  if $\frac{q'r'}{q'+r'}< 1$ (then $r>q'$),
\[
\|\Delta\varphi_\eps*u^n\|_{L^{q'}(\Omega)}\leq C\|\Delta\varphi_\eps*u^n\|_{L^{r}(\Omega)}\leq C\|D\varphi_\eps\|_{L^1(\Omega)}\|Du^n\|_{L^r(\Omega)}\leq C\eps^{-1}.
\]
Thus, overall, we get for $X_\eps:=-\nu\Delta\varphi_\eps*u^n$,
\beq\lb{4.7}
\|X_\eps\|_{L^{q'}(\Omega)}\leq C\nu\eps^{-\beta}\quad\text{ with }\beta=\max\left\{\beta_1,1\right\}(d+1)-d.
\eeq

\quad 
In view of \eqref{4.15'}, 
\[
\alpha_\eps^n-{\varphi_\eps*\alpha^{n}}/(1+C_1\eps)\leq  (X_\eps)_+.
\]
Using \eqref{F*2} and \eqref{F*3} yields
\beq\lb{4.11}
\begin{aligned}
\iint_{\Omega}F^*(x,\alpha^{n}_{\eps})dxdt&\leq \iint_{\Omega}F^*(x,\frac{\varphi_\eps*\alpha^{n}}{1+C_1\eps})+C(X_\eps)_+(1+(\alpha_\eps^n)^{q'/q})dxdt\\
&\leq \iint_{\Omega}F^*(x,\varphi_\eps*\alpha^{n})+C(X_\eps)_+(1+(\varphi_\eps*\alpha^n+(X_\eps)_+)^{\frac{q'}q})dxdt\\
&\leq \iint_{\Omega}F^*(x,\varphi_\eps*\alpha^{n})+C(X_\eps)_+^{q'}+C(X_\eps)_+(1+(\varphi_\eps*\alpha^n)^{\frac{q'}q})dxdt.
\end{aligned}
\eeq
It follows from \eqref{4.7}, H\"{o}lder's and Young's convolution inequality inequality that
\begin{align*}
   & \iint_{\Omega}(X_\eps)_+^{q'}+(X_\eps)_+(1+(\varphi_\eps*\alpha^n)^{\frac{q'}q})dxdt\\
    &\qquad\qquad\leq \|X_\eps\|_{L^{q'}(\Omega)}^{q'}+\|X_\eps\|_{L^{q'}(\Omega)}\left(\iint_\Omega 1+(\varphi_\eps*\alpha^n)^{q'}dxdt\right)^{1/q}\\
    &\qquad\qquad\leq C\nu^{q'}\eps^{-q'\beta}+C\nu\eps^{-\beta}\left(1+\|\alpha^n\|_{L^{q'}(\Omega)}^{q'/q}\right)\leq  C\nu\eps^{-\beta}
\end{align*}
where the last inequality is due to \eqref{4.3} and $\nu\eps^{-\beta}\leq 1$. 
Since $F^*(x,\cdot)$ is convex, Jensen's inequality and \eqref{F*1} show
\begin{align*}
\iint_{\Omega}F^*(x,\varphi_\eps*\alpha^n)dxdt&\leq \iint_{\Omega}\varphi_\eps*F^*(x,\alpha^{n})dxdt+C\eps(1+(\alpha^n)^{q'})dxdt\\
&\leq \iint_{\Omega}F^*(x,\alpha^{n})dxdt+C\eps,    
\end{align*}
where we applied \eqref{4.3} again.
Using these in \eqref{4.11} yields
\beq\lb{4.16}
\iint_{\Omega}F^*(x,\alpha^{n}_{\eps})dxdt\leq \iint_{\Omega}F^*(x,\alpha^{n})dxdt+C\nu\eps^{-\beta}+C\eps.
\eeq

\quad  Note that, since $\bar m\in\calC^1$, from \eqref{4.10} and \eqref{3.3} we get
\beq\lb{4.17}
\begin{aligned}
\int_{\bbT^{d}} u^n(0,x)\bar{m}(x)-u^n_\eps(0,x)\bar{m}(x)dx&\leq C\eps+\int_{\bbT^{d}} u^n(0,x)\bar{m}(x)-u^n(0,x)\frac{\varphi_\eps*\bar{m}(x)}{1+C_1\eps}dx\\
&\leq C\eps\left(1+\int_{\bbT^{d}}| u^n(0,x)|dx\right)\leq C\eps.
\end{aligned}
\eeq
Finally, we recall the definition of $\calA$ and apply \eqref{4.16} and \eqref{4.17} to get
\begin{align*}
\calA(u^{n}_{\eps},\alpha_\eps^{n})  &=\iint_{\Omega}F^*(x,\alpha_\eps^{n})dxdt-\int_{\bbT^d}u^{n}_{\eps}(0,\cdot)\bar{m}dx\\
&\leq\iint_{\Omega}F^*(x,\alpha^{n})dxdt-\int_{\bbT^d}u^{n}(0,\cdot)\bar{m}dx+C\nu\eps^{-\beta}+C\eps\\
&=\calA(u^n,\alpha^n)+C\nu\eps^{-\beta}+C\eps.
\end{align*}
\end{proof}



\section{Convergence of classical solutions -- Local Coupling}
\lb{S42}

\quad In this section, we prove the convergence of $u_\nu$ under some extra conditions (see (H6)--(H8) below). We will assume that $(u_\nu,m_\nu)$ with $\nu>0$ are classical solutions. Indeed, \cite{2012long} showed that the solutions are smooth when Hamiltonians are purely quadratic. We also refer readers to \cite{GPS1,GPS2} for results about classical solutions.


\quad We will present two results for quadratic Hamiltonians (so $q=2$). The first one is for $r>2$ and we use the conditions (H6)(H7) below (by (H6) we mean (H6-1)--(H6-3)), while the second one is for $r=2$ and $d\leq 3$, and we further need (H8).
\begin{itemize}
    \item[(H6-1)] $\bar m\in W^{2,\infty}(\bbT^d)$.

    \smallskip
    
    \item[(H6-2)] $H(x,p)$ is twice continuously differentiable in both the variables and for some $C>0$,
    \[
    |D_x^2H(x,p)|\leq C|p|^r+C\quad\text{ for all }(x,p)\in\bbT^d\times\bbR^d.
    \]

    \smallskip

    \item[(H6-3)] There exist $C\geq 1$ such that for all $x,y\in\bbT^d$ and $m,m'\geq0$,
    \[
    |f(x,m)-f(y,m)|\leq C m|x-y|,
    \]
    and
    \beq\lb{4.1.1}
    (f(x,m)-f(x,m'))(m-m')\geq \,|m-m'|^2/C.    \eeq

    \item[(H7)](Conditions on the coupling) (H1) holds with $q=2$, and for some $C>0$,
    \[
    f_m(x,m)\leq C\quad\text{ for }(x,m)\in\bbT^d\times [0,\infty).
    \]
    
    \smallskip
    
    \item[(H8)](Strict convex Hamiltonian) For for some $C\geq 1$.
    \[
    C^{-1}I_d\leq H_{pp}(x,p)\leq CI_d\quad\text{ for all $(x,p)\in\bbT^d\times\bbR^d$},
    \] 
    where $I_d$ denotes the identity matrix.
\end{itemize}
\begin{remark}
Condition {\rm (H6)} is assumed in \cite{GM}. We are going to apply their regularity results.
{\rm(H8)} implies $r=2$ in {\rm(H2)}. 
\end{remark}

\quad We start with the following lemma which says that \eqref{4.1.1} implies (H5-2) with $J_2(x,m)=m$, and (H8) implies (H5-1) with $J_1(x,p)=p$.
\begin{lemma}\lb{L.4.6}
Assume {\rm(H1)(H6)}. Then {\rm(H5-2)} holds with $J_2(x,m)=m$ and $J_2^*(x,\alpha)$ defined such that $f(x,J_2^*(x,\alpha))=\alpha$.

\quad If further assuming $H_{pp}(x,p)\geq c I_d$ for some $c>0$ uniformly in $\bbT^d\times\bbR^d$, then {\rm (H5-1)} holds with $J_1(x,p)=p$ and $J_1^*(x,\xi)$ defined such that $H_p(x,J_1^*(x,\xi))=\xi$.

\quad Here $J_2^*$ is well-defined as $f(x,\alpha)$ is strictly increasing in $\alpha$. Similarly, $J_1^*$ is well-defined as $H(x,\cdot)$ is $\calC^2 $ and is uniformly strictly convex in $p$, and the map $H_p(x,\cdot):\bbR^d\to\bbR^d$ is one-to-one. 
\end{lemma}
\begin{proof}
Let us fix $x\in\bbT^d$.
In view of Remark \ref{R.4.1}.2, it suffices to consider $m,\alpha>0$. Take $m_\alpha$ such that $f(x,m_\alpha)=\alpha$. The definition of $F^*$ yields
\beq\lb{4.1.2}
\begin{aligned}
F(x,m)+F^*(x,\alpha)-m\alpha\geq F(x,m)-F(x,m_\alpha)+(m_\alpha-m)\alpha.
\end{aligned}
\eeq
Since $F$ is strictly convex by \eqref{4.1.1}, and $F_m(x,m)=f(x,m)$, we have
\[
F(x,m)-F(x,m_\alpha)\geq f(x,m_\alpha)(m-m_\alpha)+\frac{c}{2}|m-m_\alpha|^2.
\]
Then, using $f(x,m_\alpha)=\alpha$, we get from \eqref{4.1.2} that
\[
F(x,m)+F^*(x,\alpha)-m\alpha\geq c(m-J_2^*(x,\alpha))^2/2
\]
where $J_2^*(x,\alpha):=m_\alpha$. We proved the first claim.

\quad 
The proof for the second claim is identical.
\end{proof}


\quad In the next lemma, we collect and prove some $\nu$-uniform estimates.

\begin{lemma}\lb{L.4.5}
    Assume {\rm(H2)--(H4)(H6)(H7)}, and suppose that $(m_\nu,\alpha_\nu)$ with $\nu\in [0,1]$ are weak solutions to \eqref{1.1}. Then there exists $C>0$ such that for all $\nu\in [0,1]$, 
\beq\lb{4.1.4}
\|D m_\nu\|_{L^2(\Omega)}\leq C\quad\text{and}
\eeq
\[
\|m_\nu\|_{L^\eta(\Omega)}+\|u_\nu\|_{L^{\delta}(\Omega)}\leq C,
\]
were $\eta:=\frac{2(d+1)}{d}$ and
\beq\lb{delta}
\delta:=\frac{r \eta(d+1)}{d-r(\eta-1)}\,\text{ if }\eta<1+\frac{d}{r}\quad\text{and}\quad\delta:=\infty\,\text{ if }\eta> 1+\frac{d}{r}.
\eeq
When $\eta=1+\frac{d}{r}$, $\delta>0$ can be an arbitrarily large constant, and in this case, the constant $C$ depends also on the choice of $\delta$.

\quad If further assuming $\inf_{\bbT^d\times\bbR^d}H_{pp}(x,p)\geq cI_d$ for some $c>0$, then we have
\beq\lb{4.1.5}
\|m_\nu^{1/2}D^2 u_\nu\|_{L^2(\Omega)}\leq C\quad\text{ uniformly in }\nu\in[0,1]\text{ for some $C>0$} .
\eeq
\end{lemma}
\begin{proof}
By Lemma \ref{L.4.6}, $J_1(x,\cdot)$ is an identity map. Then \eqref{4.1.4} follows from \cite[Proposition 4.3]{GM} and (H6)(H7).
In \cite{GM}, the authors only considered the case of $\nu=0$ and assumed (H5-1) with $J_1,J_1^*$ independent of $x$. However, this assumption is not needed to obtain only \eqref{4.1.4}.  Also it is not hard to see that their argument generalizes to all $\nu\in[0,1]$ and the constant is independent of $\nu$. 


\quad Note that $m_\nu\geq 0$ and $\int_{\bbT^d}m_\nu(t,\cdot)=1$. Thus it follows from \eqref{4.1.4} and the Sobolev type embedding theorem (\cite[Proposition 3.1]{dibenedettobook}) that for some dimensional constant $C>0$,
\[
\iint_{\Omega} m_{\nu}^\eta(t,x) dxdt\leq C\left(\iint_\Omega|Dm_{\nu}(t,x)|^2dxdt\right)\left(\text{ess}\sup_{t\in(0,T)}\int_{\bbT^d}m_{\nu}(t,x)dx\right)^{2/d},
\]
which is uniformly finite in $\nu$ by \eqref{4.1.4}.

\quad Next, due to (H1) and $q=2$, $f(x,m_\nu)\in L^\eta(\Omega)$ by \eqref{4.1.4}. Moreover, Proposition \ref{P.3.3} yields that $u_\nu$ is bounded from below. So we can apply Theorem 3.3 in \cite{CGPT} to get for some $C$ independent of $\nu\in [0,1]$,
\[
\|u_\nu\|_{L^{\delta}(\Omega)}\leq C\quad\text{ with }\delta \text{ given in }\eqref{delta}.
\]
It is not hard to see that when $\eta=1+\frac{d}{r}$, $\delta>0$ can be arbitrary and the constant $C$ depends also on $\delta$.

\quad The last claim follows from \cite{GM} and the second claim of Lemma \ref{L.4.6}.
\end{proof}

\quad Now we are ready to prove our main theorem of the section.

\begin{theorem}\lb{T.5.3}
Assume {\rm(H2)--(H4)(H6)(H7)} and 
\beq\lb{6.8}
r\geq \max\left\{2+\frac{d}{d+1},\,\frac{d^2}{2d+3}\right\}.
\eeq
Suppose that $(u_\nu,m_\nu)$ with $\nu\in (0,1]$ are classical solutions to \eqref{1.1}, and $(u,m)$ is a weak solution to \eqref{1.1} with $\nu=0$. Then there exists $C>0$ such that for all $\nu\in (0,1]$,   
\[
\iint_{\Omega}|u_\nu-u|^2m\,dxdt \leq C\nu^{\frac{1}{2(1+\beta)}},
\]
where $\beta$ is given in \eqref{beta}.
\end{theorem}

\begin{proof}
The proof is split into four steps. 

\quad {\bf Step 1}. For any $\nu'\in (0,\nu)$, define
\[
U(t,x):=u_{\nu}(t,x)-u_{\nu'}(t,x).
\]

Later we will pass $\nu'\to 0$. From the first equation in \eqref{1.1}, it follows that that
\begin{align}
-\partial_t U-\nu\Delta u_\nu+\nu'\Delta u_{\nu'}+H(x,Du_\nu)-H(x,Du_{\nu'})=f(x,m_\nu)-f(x,m_{\nu'}).\lb{4.2.1}   
\end{align}
Since $(u_\nu,m_{\nu})$ and $(u_{\nu'},m_{\nu'})$ are classical solutions, we can multiply \eqref{4.2.1} by $-2Um_{\nu'}$, and use $U^2$ as the test function against the equation of $m_{\nu'}$. Then adding them up and integrating from $t$ to $T$ for some $t\in [0,T)$, we find
\beq\lb{6.7}
\begin{aligned}
&\int_{\bbT^d}U^2m_{\nu'}dx\,\big|_t^T- \int_t^T \int_{\bbT^d} 2(\nu Du_\nu -\nu'Du_{\nu'}) D(Um_{\nu'})+\nu'Dm_{\nu'}D(U^2)\,dxdt\\
&\qquad\quad =\int_t^T \int_{\bbT^d} 2(H(x,Du_\nu)-H(x,Du_{\nu'}))Um_{\nu'}- D_pH(x,Du_{\nu'})D(U^2)m_{\nu'}\\
&\qquad\qquad\quad -2(f(x,m_\nu)-f(x,m_{\nu'}))U m_{\nu'}\,dxdt.
\end{aligned}
\eeq

\quad {\bf Step 2}. Since $H(x,p)$ is convex in $p$, 
\[
H(x,Du_\nu)-H(x,Du_{\nu'})-D_pH(x,Du_{\nu'})(DU)\geq 0.
\]
Also using $U(T,\cdot)=0$, \eqref{6.7} yields
\beq\lb{6.9}
\begin{aligned}
\int_{\bbT^d}U^2(t,\cdot)m_{\nu'}(t,\cdot)dx &\leq \int_t^T \int_{\bbT^d} 2(\nu Du_\nu -\nu'Du_{\nu'}) D(Um_{\nu'})+\nu'Dm_{\nu'}D(U^2)\,dxdt\\
&\quad +2\int_t^T \int_{\bbT^d} (f(x,m_\nu)-f(x,m_{\nu'}))U m_{\nu'}\,dxdt.
\end{aligned}
\eeq

\quad {\bf Step 3}. Now we estimate each term on the right-hand side of \eqref{6.9}. Recall $\eta,\delta$ from Lemma \ref{L.4.5} (with $q'=2$ since $q=2$ by (H7)). By H\"{o}lder's inequality,
\[
\int_t^T \int_{\bbT^d}(Du_{\nu}DU ) m_{\nu'}dxdt\leq \|Du_\nu\|_{L^r(\Omega)}\|DU\|_{L^r(\Omega)}\|m_{\nu'}\|_{L^{r_1}(\Omega)}
\]
where
$r_1:=\frac{2}{r-2}\leq \eta$ by \eqref{6.8}. Thus, Lemma \ref{L.4.5} and Proposition \ref{P.3.3} yield 
\[
\int_t^T \int_{\bbT^d}(Du_{\nu}DU ) m_{\nu'}dxdt\leq C
\]
for some $C$ independent of $\nu,\nu'$. Similarly, we have
$\int_t^T \int_{\bbT^d} |Du_{\nu'} DU|m_{\nu'}dxdt\leq C.$
Due to Lemma \ref{L.4.5} and \eqref{3.4} again, we get
\[
\int_t^T \int_{\bbT^d}(Du_\nu Dm_{\nu'})U\,dxdt\leq \|Du_\nu\|_{L^r(\Omega)}\|Dm_{\nu'}\|_{L^2(\Omega)}\|U\|_{L^{r_2}(\Omega)}\leq C
\]
where
$r_2:=\frac{2r}{r-2}\leq \delta$
by \eqref{6.8}. Similarly, $\int_t^T \int_{\bbT^d}Dm_{\nu'}D(U^2)dxdt$ is uniformly bounded.

\quad As for $\int_t^T \int_{\bbT^d} (f(x,m_\nu)-f(x,m_{\nu'}))U m_{\nu'}\,dxdt$, using (H7) yields
\begin{align*}
\int_t^T\int_{\bbT^d} (f(x,m_\nu)-f(x,m_{\nu'}))U m_{\nu'}\,dxdt&\leq C\iint_\Omega|m_\nu-m_{\nu'}||U|m_{\nu'}dxdt\\
    &\leq C\|m_\nu-m_{\nu'}\|_{L^2(\Omega)}\|U\|_{L^\delta(\Omega)}\|m_{\nu'}\|_{L^{r_3}}
\end{align*}
where $r_3:=\frac{2\delta}{\delta-2}$. Since $r_3\leq \eta$ by \eqref{6.8}, $\|U\|_{L^\delta(\Omega)}\|m_{\nu'}\|_{L^{r_3}}\leq C$ by Lemma \ref{L.4.5}.
Thus, Theorem \ref{T.4.3} and the first part of Lemma \ref{L.4.6} yield
\[
\iint_\Omega (f(x,m_\nu)-f(x,m_{\nu'}))U m_{\nu'}\,dxdt\leq C\nu^{\frac{1}{2(1+\beta)}}\quad\text{with $\beta$ from \eqref{beta}}.
\]

\quad {\bf Step 4}. Putting these estimates into \eqref{6.9} yields
\[
\sup_{t\in [0,T]}\int_{\bbT^d}|u_\nu-u_{\nu'}|^2 m_{\nu'}dx\leq C\nu^{\frac{1}{2(1+\beta)}}.
\]
By Theorem \ref{T.2.2}, $m_{\nu'}$ converges strongly to $m$ in $L^2(\Omega)$ as $\nu'\to 0$. Thus, as $u_\nu-u_{\nu'}$ is uniformly bounded in $L^4(\Omega)$ by Lemma \ref{L.4.5} and $\delta\geq 4$, we get
\[
\liminf_{\nu'\to 0}\iint_{\Omega}|u_\nu-u_{\nu'}|^2 m_{\nu'}dxdt=\liminf_{\nu'\to 0}\iint_{\Omega}|u_\nu-u_{\nu'}|^2 m\,dxdt.
\]

\quad Since $u_{\nu'}$ are uniformly bounded in $L^\delta(\Omega)$ and $u_{\nu'}$ converges weakly to $u$ in $L^r(\Omega)$ along a subsequence of $\nu'\to0$ by Theorem \ref{T.2.2}, we actually have $u_\nu-u_{\nu'}$ converges weakly to  $u_\nu-u$ in $L^\delta(\Omega)$ along a subsequence of $\nu'\to0$. Then we show that the functional $v\to \iint_\Omega v^2m\, dxdt$
 is continuous on $L^\delta(\Omega)$. Let $\eta'=\frac{\eta}{\eta-1}$ be the conjugate of $\eta$. Note that for any $v_1,v_2\in L^\delta(\Omega)$, by Lemma \ref{L.4.5},
 \beq\lb{111}
\iint_\Omega (v_1^2-v_2^2)m dxdt\leq C\left(\iint_\Omega (v_1^2-v_2^2)_+^{\eta'} dxdt\right)^{1/{\eta'}}\leq C\left(\iint_\Omega (v_1^{2\eta'}-v_2^{2\eta'})_+ dxdt\right)^{1/{\eta'}},
 \eeq
 and by \eqref{4.1.4},
 \[
 \delta\geq2\eta'={2\eta}/{(\eta-1)}.
 \]
 The right-hand side of \eqref{111} can be arbitrarily small when $v_1$ and $v_2$ are very close in $L^r(\Omega)$.
Thus $v\to \iint_\Omega v^2m\, dxdt$ is continuous on $L^r(\Omega)$, and it is straightforward that it is also convex.
Hence we get
\[
\iint_{\Omega}|u_\nu-u|^2m\,dxdt \leq \liminf_{\nu'\to 0}\iint_{\Omega}|u_\nu-u_{\nu'}|^2 m\,dxdt\leq C\nu^{\frac{1}{2(1+\beta)}},
\]
which finishes the proof.
\end{proof}

\quad Theorem \ref{T.5.3} fails to cover the KPZ setting \eqref{1.2} when $r=2$ that we are interested in. Below,
we prove a convergence result for $d \leq 3$ and $q=r=2$. The statement is slightly different from the previous theorem that we use weight $m_\nu$ instead of $m$.

\begin{theorem}
\label{T.5.4}
Assume {\rm(H2)--(H4)(H6)--(H8)}, $r=2$ and $d\leq 3$.  Suppose that $(u_\nu,m_\nu)$ with $\nu\in (0,1]$ are classical solutions to \eqref{1.1}, and $(u,m)$ is a weak solution to \eqref{1.1} with $\nu=0$. 
Then there exists $C>0$ such that for all $\nu\in (0,1]$,
\begin{align*}
\sup_{t\in[0,T]}\int_{\bbT^d}|u_\nu(t,x)-u(t,x)|^2m_{\nu}(t,x)dx \leq C\nu^{\frac{1}{2(1+\beta)}},
\end{align*} 
where $\beta$ is given in \eqref{beta}.
\end{theorem}
\begin{proof}
We break the proof into three steps.

\quad {\bf Step 1}. Since $d\leq 3$, Lemma \ref{L.4.5} yields that $u_\nu$ is uniformly bounded for all $\nu\in[0,1]$ in $\Omega$. 

\quad For $0<\nu'<\nu<1$, let
\[
U(t,x):=u_{\nu}(t,x)-u_{\nu'}(t,x).
\]
Then the following holds in the classical sense,
\begin{align*}
-\partial_t U-\nu\Delta u_\nu+\nu'\Delta u_{\nu'}+H(x,Du_\nu)-H(x,Du_{\nu'})=f(x,m_\nu)-f(x,m_{\nu'}).  
\end{align*}
We multiply the above equality by $-2Um_{\nu}$, and use $U^2$ as the test function against the equation of $m_{\nu}$ in \eqref{1.1}. Adding them up and integrating from $t$ to $T$ for some $t\in [0,T)$ yield
\[
\begin{aligned}
&\int_{\bbT^d}U^2m_{\nu}dx\,\big|_t^T+ \iint_\Omega 2\nu \Delta u_\nu (Um_{\nu})+\nu\nabla m_\nu \nabla(U^2)-2\nu' \Delta u_{\nu'} (Um_{\nu})\,dxdt\\
&\qquad\quad =\iint_\Omega 2(H(x,Du_\nu)-H(x,Du_{\nu'}))Um_{\nu}-D_pH(x,Du_{\nu})D(U^2)m_{\nu} \\
&\qquad\qquad\quad -2(f(x,m_\nu)-f(x,m_{\nu'}))U m_{\nu}\,dxdt.
\end{aligned}
\]
Since $U(T,\cdot)\equiv 0$, this simplifies to
\beq\lb{6.1}
\begin{aligned}
&\int_{\bbT^d}U^2(t,x)m_{\nu}(t,x)dx= \iint_\Omega 2(\nu-\nu') \Delta u_\nu (Um_{\nu})+\nu\nabla m_\nu \nabla(U^2)+2\nu' \Delta U (Um_{\nu})\\
&\qquad\qquad -\left( 2(H(x,Du_\nu)-H(x,Du_{\nu'}))Um_{\nu}-D_pH(x,Du_{\nu})D(U^2)m_{\nu} \right)\\
&\qquad\qquad +2(f(x,m_\nu)-f(x,m_{\nu'}))U m_{\nu}\,dxdt.
\end{aligned}
\eeq

\quad {\bf Step 2}. We estimate each term in the right-hand side of \eqref{6.1}. In view of the last claim of Lemma \ref{L.4.5}, by H\"{o}lder's inequality and uniform boundedness of $u,u_\nu$, we obtain
\[
\iint_\Omega (\nu-\nu')\Delta u_\nu (Um_{\nu})dxdt\leq (\nu-\nu')\|m_\nu^{1/2}\Delta u_\nu\|_{L^2(\Omega)}\|m_\nu^{1/2}U\|_{L^2(\Omega)}\leq C\nu.
\]
By Proposition \ref{P.3.3} and boundedness of $U$,
\begin{align*}
\iint_\Omega \Delta U (Um_{\nu})dxdt\leq \|\nabla U\|^2_{L^2(\Omega)}\|m_\nu\|_\infty+\|\nabla U\|_{L^2(\Omega)}\|U\|_\infty\|m_\nu\|_\infty  \leq C\|m_\nu\|_\infty.
\end{align*}
Here $\|m_\nu\|_\infty$ is finite (depending on $\nu$) because $(u_\nu,m_\nu)$ is a classical solution.
Similarly as done in the proof of Theorem \ref{T.5.3}, we have
\[
\iint_\Omega \nabla m_\nu \nabla(U^2)\,dxdt\leq C\quad \text{and}\quad \iint_\Omega (f(x,m_\nu)-f(x,m_{\nu'}))U m_{\nu}\,dxdt\leq C\nu^{\frac{1}{2(1+\beta)}}
\]
where $C$ is uniform in $\nu,\nu'$.

\quad Next, the condition (H8) and uniform boundedness of $U$ again yield
\begin{align*}
&-\frac{1}{2}\iint_\Omega 2(H(x,Du_\nu)-H(x,Du))Um_{\nu}-D_pH(x,Du_{\nu})D(U^2)m_{\nu}dxdt\\
&\qquad\qquad=\iint_\Omega(H(x,Du)-H(x,Du_\nu)-D_pH(x,Du_{\nu})(Du-Du_\nu))U m_\nu dxdt\\
&\qquad\qquad 
\leq C\iint_\Omega |Du_\nu-Du|^2 m_\nu dxdt\leq C\nu^{\frac{1}{1+\beta}},
\end{align*}
where in the last inequality, we applied Theorem \ref{T.4.3} and Lemma \ref{L.4.6}.

\quad {\bf Step 3}. Putting the above estimates together into \eqref{6.1} yields
\[
\int_{\bbT^d}U^2(t,x)m_\nu(t,x)dxdt\leq C\nu+C\nu^{\frac{1}{2(1+\beta)}}+C\|m_\nu\|_\infty\nu'.
\]
Passing $\nu'\to 0$ yields the conclusion.
\end{proof}

\section{Vanishing Viscosity Limit -- Nonlocal Coupling}\lb{S5}

\quad In this section, we discuss mean field games with cost functions that are nonlocal and are regularizing on the set of probability measures. We allow the terminal data of $u_\nu$ to depend on $m_\nu(T,\cdot)$. Consider
\beq\lb{7.1}
\left\{
\begin{aligned}
&-\partial_t u_\nu-\nu\Delta u_\nu+H(x,Du_\nu)=f(x,m_\nu(t,\cdot)),\\
&\partial_t m_\nu-\nu\Delta m_\nu-\nabla\cdot(m_\nu D_pH(x,Du_\nu))=0,\\
&m_\nu(0,x)=\bar m(x),\quad u_\nu(T,x)=\bar u(x,m_\nu(T,\cdot)).
\end{aligned}
\right.
\eeq

\quad Let $\calP$ be the set of Borel probability measures $\mu$ on $\bbT^d$. The set $\calP$ can be endowed with the well-known Kantorovitch Rubinstein distance (or $1$-Wasserstein distance), that is for any $\mu_1,\mu_2\in\calP$,
\[
\bd(\mu_1,\mu_2):=\sup_{\gamma\in\Pi(\mu_1,\mu_2)} \int_{\bbT^{2d}}|x-y|d\gamma(x,y)
\]
where $\Pi(\mu_1,\mu_2)$ denotes all  Borel probability measures on $\bbT^{d}\times\bbT^d$ that have $\mu_1$ as its first marginal and $\mu_2$ as its second marginal.
We make the following assumptions (see \cite{note,LL07}). There exists $C>0$ such that the following holds.
\begin{itemize}
    \item[(H1')] (Regularizing condition) $f:\bbT^d\times \calP\to \bbR$ satisfies for any $\mu\in\calP$,
    \[
    \|f(\cdot,\mu)\|_{\calC^2}\leq C,
    \]
    and for any $x\in\bbT^d$ and $\mu_1,\mu_2\in  \calP$,
    \[
    |f(x,\mu_1)-f(x,\mu_2)|\leq C\,\bd(\mu_1,\mu_2).
    \]
    Moreover, $f$ is strictly monotone in the second variable in the sense that for any $\mu_1,\mu_2\in\calP$, if $\mu_1\neq \mu_2$, then
    \[
    \int_{\bbT^d}(f(x,\mu_1)-f(x,\mu_2))d(\mu_1-\mu_2)(x)>0.
    \]
If $\mu\in\calP$ has a density function $m$, we write
$f(x,m):=f(x,\mu)$.
    \smallskip

    \item[(H2')] (Conditions on the Hamiltonian) Assume (H2), and that $H=H(x,p)$ is $\calC^2$ in both variables. Moreover, for all $x\in\bbT^d$ and $p\in \bbR^d$, we have
    \[
    |H_x(x,p)|\leq C(1+|p|^r),\quad |H_p(x,p)|\leq C(1+|p|^{r-1}).
    \]
    For any $R>0$, there exists $C_R>0$ such that for any $(x,p)\in\bbT^d\times\bbR^d$ with $|p|\leq R$,
    \[
    |D_{xx}H(x,p)|,\,|D_{xp }H(x,p)|\leq C_R\quad\text{and}\quad 0\leq D_{pp}H(x,p)\leq C_RI_d.
    \]
    Moreover, $H$ is convex in the second variable in the sense that there exists a non-negative function $c_1:\bbT^d\to[0,\infty)$ such that for all $p,p'\in\bbR^d$ and $x\in\bbT^d$,
\[
H(x,p)-H(x,p')-H_p(x,p)(p-p')\geq c_1(x)|p-p'|^2.
\]

    \item[(H3')]  (Conditions on the initial and terminal data) $\bar m:\bbT^d\to\bbR$ is a $\calC^1$ non-negative density function. $\bar{u}:\bbT^d\times \calP\to \bbR$ satisfies for any $\mu\in\calP$,
    \[
    \|\bar{u}(\cdot,\mu)\|_{\calC^2}\leq C, 
    \]
    and for any $x\in\bbT^d$ and $\mu_1,\mu_2\in  \calP$,
    \[
    |\bar{u}(x,\mu_1)-\bar{u}(x,\mu_2)|\leq C\,\bd(\mu_1,\mu_2).
    \]
Moreover, we have the monotonicity condition: For any $m_1,m_2\in\calP$,
    \[
    \int_{\bbT^d}(\bar u(x,\mu_1)-\bar u(x,\mu_2))d(\mu_1-\mu_2)(x)\geq0.
    \]   
If $\mu\in\calP$ has a density function $m$, we write
$\bar{u}(x,m):=\bar{u}(x,\mu)$.

\end{itemize}
\begin{remark}
1. It is a classical result (see for example, \cite{note,LL07}) that under the assumptions of {\rm(H1')(H2')(H3')}, there exists a unique classical solution to \eqref{7.1} when $\nu>0$. When $\nu=0$, \eqref{7.1} is still well-posed \cite{note,LL07}, and the first equation in \eqref{7.1} is satisfied in the viscosity sense, and the second equation holds in the weak sense. It can be shown that $m_\nu(t,\cdot)$ is continuous in time with respect to  Kantorovich-Rubinstein distance, and so $f(x,m)$ is continuous in time.

\smallskip

\quad 2. It was known that as $\nu\to 0$, the corresponding classical solutions $(u_\nu,m_\nu)$ converge uniformly to the unique solution $(u,m)$ of \eqref{7.1} with $\nu=0$ (see \cite{note}). 

\smallskip

\quad 3. The typical example of $f$ (of $\bar u$) is 
\[
f(x,m) =\int_{\bbT^d} g(y,(\phi * m)(y))\phi(x -y)dy, 
\]
where $\phi:\bbT^d\to\bbR$ is a smooth even kernel and $g:\bbT^d\to\bbR$ is a smooth function such that $g$ is strictly increasing in the second variable. Indeed, it is direct to see (see also  \cite[Example 4.1]{cannarsa2018existence})
\begin{align*}
&\int_{\bbT^d}(f(x,m_1(\cdot))-f(x,m_2(\cdot)))(m_1(x)-m_2(x))dx\\
&\qquad=\int_{\bbT^d} (g(y,(\phi * m_1)(y))-g(y,(\phi * m_2)(y)))(\phi * m_1(y)-\phi * m_2(y))dy\geq0.   
\end{align*}
\end{remark}

\quad The following regularity results are consequences of {\rm (H1')(H2')(H3')}.

\smallskip
\begin{lemma}\lb{L.5.1}
There exists $C>0$ such that for all $\nu\geq 0$ we have
\beq\lb{5.1}
m_\nu,\,|u_\nu|,\,|Du_\nu|\leq C\quad\text{ in }\Omega,
\eeq
\beq\lb{5.1'}
 \text{and }\quad D^2 u_\nu\leq CI_d\quad\text{ in }\Omega.
\eeq
\end{lemma}
\begin{proof}
The comparison principle yields that $u_\nu$ is uniformly finite for all $\nu\geq 0$. The proof for Lipschitz regularity of $u_\nu$ follows from \cite{AT} (we also refer readers to \cite{semiconcave} and \cite{tran}). Semiconcavity of $u_\nu$ is given in Theorem 5.3.6 \cite{semiconcave} (as $u_\nu$ is uniformly Lipschitz continuous, by modifying $H(x,p)$ for large $p$ one can assume that all second derivatives of $H$ are uniformly finite).
Finally, from the results in \cite[Section 4.2]{note}, it follows that $m_\nu$ is uniformly bounded.
\end{proof}

\quad Now we prove the first main theorem of the section. 

\begin{theorem}\lb{T.5.2}
Assume {\rm (H1')(H2')(H3')} and let $c_1=c_1(x)$ be from {\rm(H2')}. For $\nu\in (0,1]$, let $(u_\nu,m_\nu)$ solve \eqref{7.1} and $(u,m)$ solve \eqref{7.1} with $\nu=0$. Then there exists $C>0$ such that for all $\nu\in(0,1]$ we have
\beq\lb{6.2-1}
\iint_\Omega (f(x,m_\nu)-f(x,m))(m_\nu-m)dxdt\leq C\nu^{1/2},
\eeq
\beq\lb{6.2-2}
\int_{\bbT^d} (\bar u(x,m_\nu(T,\cdot))-\bar u(x,m(T,\cdot)))(m_\nu(T,\cdot)-m(T,\cdot))dx\leq C\nu^{1/2},
\eeq
and
\beq\lb{6.2-3}
\iint_\Omega c_1(m_\nu+m)|Du_\nu-Du|^2dxdt\leq C\nu^{1/2}.
\eeq
 
\end{theorem}
\begin{proof}
The proof consists of four steps.

\quad {\bf Step 1}. We first show that the $W^{1,2}$ norm of $\nu^{1/2} m_\nu$ is uniformly bounded for all $\nu\in (0,1]$. Note that $(u_\nu,m_\nu)$ is a classical solution \cite{note,LL07}. So we can multiply the second equation in \eqref{7.1} by $m_\nu$  and then integrate over $\Omega$ to get
\[
\int_{\bbT^2}m_\nu(T,x)^2dx-\int_{\bbT^2}m_\nu(0,x)^2dx+\nu \iint_\Omega |Dm_\nu|^2dxdt=-\iint_\Omega m_\nu D_pH(x,Du_\nu)\cdot Dm_\nu dxdt.
\]
Using \eqref{5.1} yields for some $C$ independent of $\nu$,
\[
\nu \iint_\Omega |Dm_\nu|^2dxdt\leq \frac12\iint_\Omega \nabla\cdot D_{p}H(x,Du_\nu)\, m_\nu^2\, dxdt+C.
\]
Note that $m_\nu,|Du_\nu|\leq C$ uniformly in $\nu$ by \eqref{5.1}. Thus, it follows from (H2') and \eqref{5.1'} that
\beq\lb{5.4}
\nu \iint_\Omega |Dm_\nu|^2dxdt \leq C\iint_\Omega  |D_{xp}H(x,Du_\nu)|+ \left[\tr \left(D_{pp}H(x,Du_\nu)D^2u_\nu \right)\right]_+\, dxdt+C\leq C 
\eeq
for some constant $C$ independent of $\nu\in (0,1]$.

\quad {\bf Step 2}. Now we take any $\nu'\in (0,\nu)$, and define
\[
U(t,x):=u_{\nu}(t,x)-u_{\nu'}(t,x),\quad M(t,x):=m_\nu(t,x)-m_{\nu'}(t,x).
\]
It follows from the equations in \eqref{7.1} that
\begin{align}
&-\partial_t U-\nu\Delta u_\nu+\nu'\Delta u_{\nu'}+H(x,Du_\nu)-H(x,Du_{\nu'})=f(x,m_\nu)-f(x,m_{\nu'}),\lb{5.2}\\
&\partial_t M-\nu\Delta m_\nu+\nu'\Delta m_{\nu'}-\nabla\cdot(m_\nu D_pH(x,Du_\nu))+\nabla\cdot(m_{\nu'} D_pH(x,Du_{\nu'}))=0.\lb{5.3}   
\end{align}
Let us multiply \eqref{5.2} by $-M$, and \eqref{5.3} by $U$, and integrate over $\Omega$. We can do these as $(u_\nu,m_{\nu})$ and $(u_{\nu'},m_{\nu'})$ are classical solutions. Since $M(0,\cdot)=0$, we get
\beq\lb{5.5}
\begin{aligned}
&\int_{\bbT^d}U(T,x)M(T,x)dx-\iint_\Omega \nu Du_\nu DM-\nu'Du_{\nu'}DM-\nu Dm_{\nu}DU+\nu'Dm_{\nu'}DU\,dxdt\\
&\qquad\quad =\iint_\Omega (H(x,Du_\nu)-H(x,Du_{\nu'}))M+(m_\nu D_pH(x,Du_\nu)-m_{\nu'} D_pH(x,Du_{\nu'}))DU\\
&\qquad\qquad\quad -(f(x,m_\nu)-f(x,m_{\nu'}))M\,dxdt\\
&\qquad\quad =-\iint_\Omega m_{\nu}(H(x,Du_{\nu'})-H(x,Du_{\nu})-D_pH(x,Du_{\nu})(Du_{\nu'}-Du_{\nu}))\\
&\qquad\qquad\quad +
m_{\nu'}(H(x,Du_{\nu})-H(x,Du_{\nu'})-D_pH(x,Du_{\nu'})(Du_{\nu}-Du_{\nu'}))
\\
&\qquad\qquad\quad+(f(x,m_\nu)-f(x,m_{\nu'}))(m_\nu-m_{\nu'})\,dxdt\\
&\qquad\quad\leq -\iint_\Omega c_1(x)(m_{\nu}+m_{\nu'})|Du_\nu-Du_{\nu'}|^2+(f(x,m_\nu)-f(x,m_{\nu'}))(m_\nu-m_{\nu'})\,dxdt,
\end{aligned}
\eeq
where we used (H2') in the inequality.

\quad {\bf Step 3}. Now we estimate the several terms in \eqref{5.5}. By direct computations,
\begin{align*}
 &\iint_\Omega\nu Du_\nu DM-\nu'Du_{\nu'}DM-\nu Dm_{\nu}DU+\nu'Dm_{\nu'}DU\,dxdt\\
 &\qquad\quad=\iint_\Omega(\nu-\nu')(D u_{\nu'} D m_\nu-D u_\nu Dm_{\nu'})dxdt=\iint_\Omega(\nu-\nu')(D u_{\nu'} D m_\nu+\Delta u_\nu m_{\nu'}) dxdt  .
\end{align*}
By \eqref{5.1} and \eqref{5.1'}, $u_\nu$ and $u_{\nu'}$ are uniformly Lipschitz continuous and semiconcave in space, and $m_{\nu'}$ is uniformly bounded. Combining these with $\nu'<\nu$ and \eqref{5.4}, we obtain
\beq\lb{5.6}
\begin{aligned}
&    \iint_\Omega\nu Du_\nu DM-\nu'Du_{\nu'}DM-\nu Dm_{\nu}DU+\nu'Dm_{\nu'}DU\,dxdt\\
&\qquad\qquad\leq C\nu \iint_\Omega|D m_{\nu}|+m_{\nu'}dxdt\leq C\nu \left(\iint_\Omega|D m_{\nu}|^2dxdt\right)^{1/2}+C\nu\leq C\nu^{1/2}.
\end{aligned}
\eeq

\quad Thus \eqref{5.5} and \eqref{5.6} yield
\begin{align*}
&\int_{\bbT^d}U(T,x)M(T,x)dx+\iint_\Omega c_1(x)(m_{\nu}+m_{\nu'})|Du_\nu-Du_{\nu'}|^2dxdt\\
&\qquad\qquad\qquad+\iint_\Omega (f(x,m_\nu)-f(x,m_{\nu'}))(m_\nu-m_{\nu'})dxdt\leq C\nu^{1/2}
\end{align*}
In view of (H2')(H3'), the three terms on the left-hand side of the above are non-negative. Thus passing $\nu'\to 0$ yields \eqref{6.2-1} and \eqref{6.2-2}.


\quad {\bf Step 4}. Finally, we pass $\nu'\to 0$ in
\[
\iint_\Omega c_1(m_{\nu}+m_{\nu'})|Du_\nu-Du_{\nu'}|^2dxdt\leq C\nu^{1/2}.
\]    
Because $u_{\nu'}$ converges to $u$ locally uniformly and $u_{\nu'}$ is uniformly semi-concave in $x$ for all $\nu'\geq 0$, $Du_{\nu'}(t,\cdot)$ converges to $Du(t,\cdot)$ a.e. $x\in\bbT^d$. Thus passing $\nu'$ to $0$ yields \eqref{6.2-3}.
\end{proof}

\section{Convergence of $u_\nu$  -- Nonlocal Coupling} 
\label{S61}
\quad In this section, we proceed to show convergence results of $u_\nu$ as $\nu\to0$. 
First, let us assume a strong condition (H4') on $f$ and $\bar u$, and we prove pointwise convergence of $u_\nu$ to $u$. Later, we also consider a weaker condition (H4'').
\begin{itemize}
    \item[(H4')] There exists $C>0$ such that for any $\eps>0$, if $\mu_1,\mu_2\in\calP$ satisfy
    \[
\int_{\bbT^d} (f(x,\mu_2)-f(x,\mu_1))d(\mu_2-\mu_1)(x)\leq \eps,
\]  
then 
\[
\sup_{x\in\bbT^d}|f(x,\mu_2)-f(x,\mu_1)|\leq C\eps^{1/2}.
\]
And the same holds if we replace $f$ by $\bar u$.
\end{itemize}

\begin{remark}
  Note that if $f$ (or $\bar u$) is of the form
$f(x,m) =\int_{\bbT^d} g(y,(\phi * m)(y))\phi(x -y)dy$, where $\phi$ is a smooth even function on $\bbT^d$ and $g$ is $\calC^1$ on $\bbT^d\times\bbR$,
then the condition reduces to
\[
\text{for all $(x,z)\in \bbT^d\times \bbR$ we have }g_z(x,z)>0.
\]
Indeed suppose $m_1,m_2$ are two probability density function in $\bbT^d$. Since $|\phi*m_i|\leq\|\phi\|_\infty$, we have
\[
|f(x,m_2)-f(x,m_1)|\leq \|\phi\|_\infty\int_{\bbT^d} |g(y,\phi*m_1)-g(y,\phi*m_2)|dy\leq C\|\phi*m_1-\phi*m_2\|_{L^2(\bbT^d)}.
\]
On the other hand, for $c:=\inf_{(x,z)\in\bbT^d\times [-\|\phi\|_\infty,\|\phi\|_\infty]}g_z(x,z)>0$, we get
\[
\int_{\bbT^d} (f(x,m_2)-f(x,m_1))(m_2-m_1)dx\geq c \int_{\bbT^d} |\phi*m_1-\phi*m_2|^2 dx.
\]
So there exists $C>0$ independent of $m_1,m_2$ such that for all $x\in\bbT^d$,
\[
|f(x,m_2)-f(x,m_1)|^2\leq C\int_{\bbT^d} (f(x,m_2)-f(x,m_1))(m_2-m_1)dx.
\]  
\end{remark}

\quad We will need the following lemma.
\begin{lemma}\lb{L.6.4}
Let $\eps>0$, and let $H$ be as before, and $g:[0,T]\times\bbT^d\to\bbR$ be $\calC^1$ in space and Lipschitz continuous in time. Suppose that $v$ and $v_\eps$ are Lipschitz continuous and satisfy
\[
\partial_t v+H(x,Dv)=g(t,x)\quad\text{and}\quad\partial_t v_\eps-\eps\Delta v_\eps+H(x,Dv_\eps)=g(t,x)
\]
in the sense of viscosity, and $v(0,\cdot)=v_\eps(0,\cdot)$ is a $\calC^2$ function. Then there exists $C>0$ such that for all $\eps>0$, 
\[
|v(t,x)-v_\eps(t,x)|\leq C\eps^{1/2}\quad\text{in }[0,T]\times\bbT^d.
\]
\end{lemma}
\quad Though the proof is given by the classical viscosity solution approach (see e.g., \cite{two} for the case when $H=H(p)$ and $g\equiv 0$), for readers' convenience, we provide it in the appendix.

\begin{theorem}\lb{T.6.3'}
Under the assumptions of Theorem \ref{T.5.2}, assume {\rm(H4')}. Then there exists $C>0$ such that for all $\nu\in (0,1]$ we have
\[
\sup_{(t,x)\in [0,T]\times \bbT^d}|u(t,\cdot)-u_\nu(t,\cdot)|\leq C\nu^{1/4}.
\]
\end{theorem}
\begin{proof}
For $\nu\in (0,1]$, let $w_\nu$ be the unique solution to
\[
-\partial_t w_\nu-\nu\Delta w_\nu+H(x,Dw_\nu)=f(x,m)\quad\text{ with }w_\nu(T,x)=\bar{u}(x,m(T,\cdot)).
\]    
The condition (H4') and Theorem \ref{T.5.2} imply that
\[
|f(x,m)-f(x,m_\nu)|,\quad |w_\nu(T,x)-u_\nu(T,x)|=|\bar u(x,m(T,\cdot))-\bar u(x, m_\mu(T,\cdot))|\leq C\nu^{1/4}.
\]
Thus by comparing $u_\nu$ with $w_\nu\pm C\nu^{1/4}(T-t+1)$ yields that
\[
|u_\nu-w_\nu|\leq C\nu^{1/4}\quad\text{ in }\Omega.
\]

\quad 
Note $u$ satisfies
\[
-\partial_t u+H(x,Du)=f(x,m)\quad\text{ with }w(T,x)=\bar{u}(x,m(T,\cdot)),
\]
in the sense of distribution, and $u$ is also a viscosity solution by \cite{16}. 
Moreover, by \cite[Lemma 4.14]{note} and (H1'),
$f(x,m)$ is Lipschitz continuous in time. 
Then it follows from Lemma \ref{L.6.4} that
\[
\sup_{(t,x)\in \Omega}|w_{\nu}(t,x)-u(t,x)|\leq C\nu^{1/2},
\]
which finishes the proof.
\end{proof}

\quad Now we consider a weaker condition:
\begin{itemize}
    \item[(H4'')] There exists $C>0$ such that for any $\eps>0$, if $\mu_1,\mu_2\in\calP$ satisfy
    \[
\int_{\bbT^d} (f(x,\mu_2)-f(x,\mu_1))d(\mu_2-\mu_1)(x)\leq \eps,
\]  
then 
\[
\int_{\bbT^d}|f(x,\mu_2)-f(x,\mu_1)|dx\leq C\eps^{1/2}.
\]
And the same holds if we replace $f$ by $\bar u$ .
\end{itemize}

\quad Since, under (H4''), it is only known that $f(x,m(t,\cdot))$ and $f(x,m_\nu(t,\cdot))$ are close in average from Theorem \ref{T.5.2}, we do not expect a pointwise strong convergence of $u_\nu$ to $u$. We will apply a dual equation method to prove an $L^\infty_t L^1_x$ convergence (see also \cite{lin}).

\begin{theorem}\lb{T.6.3}
Under the assumptions of Theorem \ref{T.5.2}, assume {\rm(H4'')}. Then there exists $C>0$ such that for all $\nu\in (0,1]$ we have
\[
\sup_{t\in [0,T]}\|u(t,\cdot)-u_\nu(t,\cdot)\|_{L^1(\bbT^d)}\leq C\nu^{1/4}.
\]
\end{theorem}
\begin{proof}
\quad For $\nu\in (0,1]$, let $w_\nu$ be the unique solution to
\[
-\partial_t w_\nu-\nu\Delta w_\nu+H(x,Dw_\nu)=f(x,m)\quad\text{ with }w_\nu(T,x)=\bar{u}(x,m(T,\cdot)).
\]
By Lemma \ref{L.5.1}, there exists $C>0$ such that for all $\nu\in( 0,1]$, 
\beq\lb{7.10}
|Du|,\,|Du_\nu|,\,|w_\nu|\leq C\quad\text{and}\quad D^2u_\nu,D^2w\leq CI_d\quad\text{in }\Omega.
\eeq
We will compare $u_\nu$ with $w_\nu$ and then $w_\nu$ with $u$. 


\quad {\bf Step 1}. Consider $W:=w_{\nu}-u_\nu$, which then satisfies
\beq\lb{7.11}
-\partial_t W-\nu\Delta W+G\cdot DW=f(x,m)-f(x,m_\nu)
\eeq
with 
\[
W(T,x)=w_{\nu}(T,x)-u_\nu(T,x)=\bar u(x,m(T,\cdot))-\bar u(x,m_\nu(T,\cdot)),
\]
and
\[
G(t,x):=\int_0^1 D_pH(x,s Dw_{\nu}+(1-s)Du_\nu)ds.
\]
Recall that
\[
-\partial_t u_\nu-\nu\Delta u_\nu+H(x,Du_\nu)=f(x,m_\nu)\quad\text{ with }w_{\nu}(T,x)=\bar{u}(x,m_\nu(T,\cdot)).
\]
By (H4'') and Theorem \ref{T.5.2},
\beq\lb{7.2}
\begin{aligned}
\iint_\Omega|f(x,m)-f(x,m_\nu)|dxdt&\leq \int_0^tC\left(\int_{\bbT^d} (f(x,m)-f(x,m_\nu))(m-m_\nu)dx\right)^{1/2}dt\\
&\leq C\left(\iint_\Omega (f(x,m)-f(x,m_\nu))(m-m_\nu)dxdt\right)^{1/2}
\leq C\nu^{1/4}.
\end{aligned}
\eeq
Similarly, by \eqref{6.2-2} of Theorem \ref{T.5.2},
\beq\lb{7.3}
\int_{\bbT^d}|\bar{u}(x,m(T,\cdot))-\bar{u}(x,m_\nu(T,\cdot))|dx\leq C\nu^{1/4}.
\eeq

\quad 

\quad {\bf Step 2}. For any fixed $t_1\in [0,T)$, we consider the dual equation of \eqref{7.11} in ${[t_1,T]\times\bbT^d}$:
\[
\psi_t-\nu\Delta\psi-\nabla\cdot (G\psi)=0\quad\text{with }\psi(t_1,\cdot)=\psi_0,
\]
where $\psi_0$ is a smooth function on $\bbT^d$. By Divergence Theorem, we have
\[
\frac{d}{dt}\int_{\bbT^d}W\psi dx=-\int_{\bbT^d}(f(x,m)-f(x,m_\nu))\psi dx.
\]
In view of \eqref{7.2} and \eqref{7.3}, integrating the above equality in the time interval $[t_1,T]$ yields
\beq\lb{7.4}
\begin{aligned}
\left|\int_{\bbT^d} W(t_1,x)\psi_0(x)dx\right|&\leq \left|\int_{\bbT^d} W(T,x)\psi(T,x)dx\right|+C\nu^{1/4}\|\psi\|_{L^\infty({[t_1,T]\times\bbT^d})}\\
&\leq C\nu^{1/4}\|\psi\|_{L^\infty({[t_1,T]\times\bbT^d})}.  
\end{aligned}
\eeq

\quad {\bf Step 3}. Now, we estimate $
\|\psi\|_{L^\infty({[t_1,T]\times\bbT^d})}$ in terms of $\|\psi_0\|_{L^\infty(\bbT^d)}$. It follows from the equation of $\psi$ that for any $n\geq 2$ an even number,
\[
\frac{d}{dt}\int_{\bbT^d} \psi^ndx+n(n-1)\nu \int_{\bbT^d}|D\psi|^2 \psi^{n-2}dx=-n(n-1)\int_{\bbT^d} (G\cdot D\psi)\,\psi^{n-1}.
\]
So for any $t\in [t_1,T]$,
\[
\int_{\bbT^d} \psi(t,x)^ndx\leq \|\psi_0\|_{\infty}^n+(n-1)\int_{t_1}^t\int_{\bbT^d} (\nabla\cdot G(s,x))\psi^{n}(s,x)dxds.
\]
This yields that if $\nabla\cdot G(s,x)$ is uniformly bounded from above for all $\nu$, then by Gronwall's inequality,
\[
\|\psi(t,\cdot)\|_{L^n(\bbT^d)}\leq \|\psi_0\|_{L^\infty(\bbT^d)}\exp\left(\|(\nabla\cdot G)_+\|_{L^\infty({[t_1,T]\times\bbT^d})}(t-t_1)\right).
\]
Passing $n\to\infty$ yields
\[
\|\psi\|_{L^\infty({[t_1,T]\times\bbT^d})}\leq \|\psi_0\|_{L^\infty(\bbT^d)}\exp\left(\|(\nabla\cdot G)_+\|_{L^\infty({[t_1,T]\times\bbT^d})}T\right).
\]
Applying this in \eqref{7.4}, we get for some $C>0$,
\[
\left|\int_{\bbT^d} W(t_1,x)\psi_0(x)dx\right|\leq C\nu^{1/4}\|\psi_0\|_{L^\infty(\bbT^d)}
\]
holds for all smooth $\psi_0$, which implies that
\beq\lb{7.5}
\|w_{\nu}(t_1,\cdot)-u_\nu(t_1,\cdot)\|_{L^1(\bbT^d)}\leq C\nu^{1/4}\quad\text{ for any }t_1\in [0,T).
\eeq

\quad {\bf Step 4}. To finish the proof of \eqref{7.5}, we now show that $\nabla\cdot G(t,x)$ is uniformly bounded from above. Indeed we have
\begin{align*}
   \nabla\cdot G&=\int_0^1 \sum_i\frac{\partial^2 H}{\partial p_i\partial x_i}+\left(s\sum_{i,j}\frac{\partial^2 H}{\partial p_i\partial p_j }\frac{\partial^2 w_{\nu}}{\partial x_i\partial x_j }+(1-s)\sum_{i,j}\frac{\partial^2 H}{\partial p_i\partial p_j }\frac{\partial^2 u_\nu}{\partial x_i\partial x_j }\right)ds.
\end{align*}
By \eqref{7.10} and (H2'), there exists $C>0$ such that for all $\nu,s\in [0,1]$,
\[
\frac{\partial^2 H}{\partial p_i\partial x_i}(x,s Dw_{\nu}+(1-s)Du_\nu)\leq C,
\]
\[
\frac{\partial^2 H}{\partial p_i\partial p_j },\quad\frac{\partial^2 w_{\nu}}{\partial x_i\partial x_j },\quad \frac{\partial^2 u_\nu}{\partial x_i\partial x_j }\leq C.
\]
Therefore, we obtain that $\nabla\cdot G\leq C$ for some $C$ independent of $\nu$.

\quad {\bf Step 5}. Due to \eqref{7.5}, 
to conclude the theorem, it suffices to show that
\[
\sup_{(t,x)\in \Omega}|w_{\nu}(t,x)-u(t,x)|\leq C\nu^{1/2}.
\]
This is a consequence of Lemma \ref{L.6.4} (see also the proof of Theorem \ref{T.6.3'}).
\end{proof}

\section{Open problems}
\label{S8}

\quad In this section, we collect a few open problems related to the convergence of vanishing viscosity approximations for MFGs
and its connection to the KPZ equation.

\begin{enumerate}[itemsep = 3 pt]
\item
In Theorem \ref{T.4.3}, we proved a rate of $\nu^{\frac{1}{2(1+\beta)}}$ for the convergence of $(m_\nu, Du_\nu)$ in some Sobolev norm 
assuming that $H(\cdot,p)$ grows as $|p|^r$, and $f(\cdot,m)$ grows as $m^{r-1}$. 
We know that if $\frac{1}{q} + \frac{1}{r} \le 1$, this rate is $\nu^{\frac{1}{4}}$ which is independent of the dimension $d$;
while $\frac{1}{q} + \frac{1}{r} > 1$, the exponent $\frac{1}{2(1+\beta)} \asymp d^{-1}$.
Is the rate $\nu^{\frac{1}{2(1+\beta)}}$ is tight, 
or does the threshold $\frac{1}{q} + \frac{1}{r} =1$ induces a phase transition in
the dimension dependence of the rate exponent of vanishing viscosity approximations for MFGs?

\smallskip

\item
In Theorems \ref{T.4.3}, \ref{T.5.3} and \ref{T.5.4}, we proved the convergence (rate) of $(m_\nu, u_\nu)$ in $L^2$ norm. 
For the (1+1)-dimensional KPZ equation, it is known that under narrow wedge initial condition, 
$u_\nu$ converges locally uniformly but without rate.
\begin{enumerate}[itemsep = 3 pt]
\item
Does the convergence still hold in a stronger sense, e.g. locally uniformly for $u_\nu$,
and do we get the same convergence rate?

\smallskip

\item
Can we relax the condition $q=2$, \eqref{6.8} and $r=2$ in these theorems?
\end{enumerate}

\smallskip
\noindent
In Theorem \ref{T.5.4}, we proved the convergence of $u_\nu$ in $L^2$ norm weighted by $m_\nu$.
Can we prove the same rate in $L^2$ norm weighted by $m$ (and further locally uniformly)?

\smallskip

\item
For the $(1+1)$-dimensional KPZ equation, Theorem \ref{T.4.3} implies that
$\rho_\nu$ converges in $L^2$ norm with a rate $\nu^{\frac{1}{4}}$;
Theorem \ref{T.5.4} shows that $h_\nu$ converges in a weighted $L^2$ norm with rate $\nu^{\frac{1}{8}}$.
This relies on the Cole-Hopf transform from the SHE to the KPZ equation, which underlies the weak noise theory.
Is there a higher dimensional weak noise theory to connect large deviations of the KPZ equation to 
vanishing viscosity for MFGs?

\smallskip

\item 
There are also a few directions to extend this work from a PDE perspective.
For instance,
\begin{enumerate}[itemsep = 3 pt]
\item[(a)]
Is it possible to generalize our results to non-compact domains (e.g. $\mathbb{R}^d$)?
\item[(b)]
In the local coupling setting, can we allow the terminal data $u(T,\cdot)$ to depend on $m(T,\cdot)$ pointwise? 
\item[(c)]
What is the convergence rate of vanishing viscosity approximations for kinetic MFGs (see \cite{griffin2022variational})?
\end{enumerate}

\end{enumerate}

\quad We hope that our convergence analysis of vanishing viscosity approximations for MFGs may trigger further research
connecting the KPZ equation and MFGs in dimension $d \ge 2$,
and on the dimension effect as $d \to \infty$.

\bigskip
{\bf Acknowledgments}: 
We thank Sharon Xuan Di for various pointers to the applications of first-order MFGs in transportation,
and Yier Lin for telling us about the weak noise theory for the KPZ equation.
We thank Andrzej \'{S}wi\c{e}ch for remarks which lead to the first part of Section \ref{S61}, and Alp{\'a}r R. M{\'e}sz{\'a}ros for various pointers to references.
We also thank Ivan Corwin, Fran\c{c}ois Delarue, Dan Lacker, Qiang Du and Jianfeng Zhang for helpful discussions.
The work of W. Tang is supported by NSF grants DMS-2113779 and DMS-2206038, and a start-up grant at Columbia University.

\appendix

\section{}
\begin{lemma}
    Suppose $q,q'>1$ satisfying  $\frac{1}{q}+\frac{1}{q'}=1$, and $\sigma>0$. Then for $c_0:=(\max\{q,q'\})^{-1}$, we have for all $m,\alpha\geq 0$,
    \[
    \frac{1}q (\sigma m)^q+\frac{1}{q'}(\alpha/\sigma)^{q'}\geq m\alpha+c_0((\sigma m)^{q/2}-(\alpha/\sigma)^{q'/2})^2.
    \]
\end{lemma}
\begin{proof}
By replacing $\sigma m$ by $m$, and $\alpha/\sigma$ by $\alpha$, we can assume without loss of generality that $\sigma=1$. 
Also we can assume that $m,\alpha>0$ and $q>2>q'>1$, otherwise the proof is direct. 
In this case, $c_0=\frac{1}{q}$. By direct computations, the inequality is reduced to
\beq\lb{A1}
(\frac1{q'}-\frac1q)\alpha^{q'-1}+\frac2{q}m^{q/2}\alpha^{q'/2-1}\geq m.
\eeq
By H\"{o}lder's inequality, 
\[
\gamma_1\alpha^{q'-1}+\gamma_2 m^{q/2}\alpha^{q'/2-1}\geq  \alpha^{\gamma_1(q'-1)+\gamma_2(q'/2-1)}m^{\gamma_2(q/2)}=m,
\]
where
$\gamma_1:=1-2/q,\gamma_2:=2/q$. This implies \eqref{A1} immediately, which finishes the proof.

\quad It is also not hard to see from the proof that $c_0$ can not be larger than $\frac1q$. Otherwise, if $c_0>\frac1q$, we need 
\[
(\frac1{q'}-c_0)\alpha^{q'-1}+\frac2{q}m^{q/2}\alpha^{q'/2-1}\geq m+(c_0-\frac1q)m^q/\alpha,
\]
but this cannot be true as passing $\alpha\to0$ leads to a contradiction.
\end{proof}

\section{Proof of Lemma \ref{L.6.4}}
\quad Let us only prove that $ \sup_{ \Omega } (v- v_\eps)$ is small. The other side of the estimate follows the same. 
We split the proof into five steps.

\quad {\bf Step 1}. First of all, since $v,v_\eps$ are uniformly Lipschitz continuous, by modifying $H(x,p)$ in the region where $|p|$ is large (see e.g., section 5 of \cite{tang2023policy}), we can further assume without loss of generality that for some $C>0$,
\beq\lb{B1}
|H_x(x,p)|\leq C\quad\text{ for all }(x,p)\in\bbT^d\times\bbR^d.
\eeq

\quad Let us take $T\geq 1$. Suppose $(t_0,x_0)\in [0,T]\times \bbT^d$ is such that
\beq\lb{B3.2}
3\sigma:=v(t_0,x_0)-v_\eps (t_0,x_0)= \sup_{(t,x)\in \Omega}\left[ v(t,x)-v_\eps (t,x)\right]>0.
\eeq
Below we will show $\sigma\leq C\sqrt{\eps}$ where $C$ depends on $T$, the Lipschitz constant of $v$ and $v_\eps$, and the assumptions.

\quad {\bf Step 2}. Consider a smooth function $\varphi:\bbR^{d+1}\to [0,1]$ such that
\begin{enumerate}
    \item[(i)] $\varphi(t,x)=1-t^2-|x|^2$ if $t^2+|x|^2<1/2$,
    
    \smallskip
    
    \item[(ii)] $0\leq \varphi(t,x)\leq 1/2$ if $t^2+|x|^2>1/2$, and $\varphi(t,x)= 0$ if $t^2+|x|^2>1$.
\end{enumerate}
For $\delta>0$, denote $\varphi_\delta(t,x):=\varphi(t/\delta,x/\delta)$, $\Omega:=(0,T)\times\bbT^d$, and
\[
 L:=\sup\left\{v(t,x),-v_\eps (t,x)\,:\,(t,x)\in \Omega\right\}+1\geq 1,
\]

\quad Next, we define $\Phi_{\eps,\delta}: [0,T]^2\times\bbT^{2d}\to \bbR$ by
\begin{align*}
\Phi_{\eps,\delta}(t,s,x,y):=v(t,x)-v_\eps (s,y)-{\sigma}(2T-t-s)/T+8L\varphi_\delta(t-s,x-y).
\end{align*}
There exists $(t_1,s_1,x_1,y_1)\in [0,T]^2\times \bbT^{2d}$ such that
\beq\lb{B3.3}
\Phi_{\eps,\delta}(t_1,s_1,x_1,y_1)=\max_{[0,T]^2\times \bbT^{2d}}\Phi_{\eps,\delta}(t,s,x,y),
\eeq
and by \eqref{B3.2} and $\varphi_\delta(0,0)=1$,
\beq\lb{3.10}
\Phi_{\eps,\delta}(t_1,s_1,x_1,y_1)\geq \Phi_{\eps,\delta}(t_0,t_0,x_0,x_0)\geq 8L+\sigma.
\eeq
Since $\max\{v(t_1,x_1),-v_\eps (s_1,y_1)\}\leq L$, 
\[
\Phi_{\eps,\delta}(t_1,s_1,x_1,y_1)\leq 2L+8L\varphi_\delta(t_1-s_1,x_1-y_1),
\]
which, together with \eqref{3.10}, implies
$\varphi_\delta(t_1-s_1,x_1-y_1)\geq {3}/{4}$.
Then by (i), we get 
\beq\lb{B3.4}
\varphi_\delta(t-s_1,x-y_1)=1-(|t-s_1|^2+|x-y_1|^2)/\delta^2,
\eeq
whenever $|t-t_1|,|x-x_1|\leq \delta/C$ for some $C>0$.

\quad {\bf Step 3}. Now, in view of \eqref{B3.3}, the mapping
\beq\lb{B3.5}
(t,x)\mapsto v(t,x)+{\sigma}t/T+8L\varphi_\delta(t-s_1,x-y_1).
\eeq
is maximized at $(t,x)=(t_1,x_1)$. 
As $v$ is uniformly Lipschitz continuous by the assumption, 
we get that 
\[
|D\varphi_\delta(t_1-s_1,x_1-y_1)|\leq C/L\quad\text{ and }
\] 
\[
|\partial_t \varphi_\delta(t_1-s_1,x_1-y_1)|\leq C(1+\sigma/T)/L.
\]
By \eqref{B3.4} and $\sigma\leq L$, these yield
\beq\lb{B3.8}
|x_1-y_1|\leq C\delta^2/L\quad\text{ and }
\eeq
\beq\lb{3.9}
|t_1-s_1|\leq C\delta^2(1+\sigma/T)/L.
\eeq

\quad {\bf Step 4}. We firstly assume that $t_1,s_1>0$. In view of \eqref{B3.5}, the viscosity solution test for $v$ yields
\[
\begin{aligned}
-{\sigma}/{T}-8L\,\partial_t \varphi_\delta(t_1-s_1,x_1-y_1)+H\left(x_1,-8L\,D \varphi_\delta(t_1-s_1,x_1-y_1))\right)\geq g(t_1,x_1).
\end{aligned}
\]
Similarly, because
\[
(s,y)\to v_\eps(s,y)-\frac{\sigma}{T}s-8L\,\varphi_\delta(t_1-s,x_1-y)
\]
is minimized at $(s_1,y_1)$, the viscosity solution test yields
\begin{align*}
&{\sigma}/{T}-8L\,\partial_t \varphi_\delta(t_1-s_1,x_1-y_1)+H\left(y_1,-8L\, D \varphi_\delta(t_1-s_1,x_1-y_1)\right)\\
&\qquad\qquad -8\eps L\Delta \varphi_\delta(t_1-s_1,x_1-y_1)\leq g(s_1,y_1).
\end{align*}
Thus we get
\beq\lb{3.11}
\begin{aligned}
{2\sigma}/{T}&\leq H\left(x_1,-8L\,D \varphi_\delta(t_1-s_1,x_1-y_1)\right)-H\left(y_1,-8L\,D \varphi_\delta(t_1-s_1,x_1-y_1)\right)\\
&\qquad\qquad + 8\eps L\,\Delta\varphi_\delta(t_1-s_1,x_1-y_1)+g(s_1,y_1)-g(t_1,x_1).
\end{aligned}
\eeq
For the second order term, the definition of $\varphi_\delta$ implies
\beq\lb{3.6}
8\eps L\,\Delta \varphi_\delta(t_1-s_1,x_1-y_1)\leq C\eps L\delta^{-2}.
\eeq
Using \eqref{3.11}, \eqref{3.6}, \eqref{B1} and Lipschitz continuity of $g$ yields for some universal $C$,
\begin{align*}
{2\sigma}/{T}\leq C\eps L\delta^{-2}+C(|x_1-y_1|+|t_1-s_1|)\leq C\eps L\delta^{-2} +C\delta^2(1+\sigma/T)/L
\end{align*}
where in the last inequality we used \eqref{B3.8} and \eqref{3.9}.
Finally, taking $\delta:=L^{1/2}\eps^{1/4}$ yields
$\sigma\leq  C\sqrt{\eps}$ when $\eps$ is sufficiently small depending only on $T,C$.
This finishes the proof of the upper bound of $\sup_{ \Omega }(v-v_\eps )$ in the case when $t_1,s_1>0$.

\quad {\bf Step 5}. Finally, suppose that one of $t_1$ and $s_1$ equals to $T$. Let us only prove for the case when $t_1=0$. By \eqref{3.10},
\[
8L+\sigma \leq \Phi_{\eps,\delta}(t_1,s_1,x_1,y_1)\leq v(t_1,x_1)-v_\eps (s_1,y_1)+8L\,\varphi_\delta(t_1-s_1,x_1-y_1).
\]
Again using that $v_\eps $ is uniformly Lipschitz continuous, this, \eqref{B3.8} and \eqref{3.9} yield
\begin{align*}
8L+\sigma
&\leq |v(0,x_1)-v(0,y_1)|+|v_\eps(0,y_1)-v_\eps (s_1,y_1)|+8L\,\varphi_\delta(-s_1,x_1-y_1)\\
&\leq C(|x_1-y_1|+|s_1|)+8L \leq C\delta^2(1+\sigma/T)/L+8L.
\end{align*}
Since $\delta=L^{1/2}\eps^{1/4}$, this yields $\sigma\leq C \sqrt{\eps}$ for some universal $C>0$ when $\eps$ is small.

\bibliography{mfg}

\begin{thebibliography}{10}

\bibitem{achdou2021mean}
Y.~Achdou, P.~Cardaliaguet, F.~Delarue, A.~Porretta, and F.~Santambrogio.
\newblock {\em Mean field games}, volume 2281 of {\em Lecture Notes in
  Mathematics}.
\newblock Springer, Cham, 2020.
\newblock Notes from the CIME School held in Cetraro, June 2019, Edited by
  Cardaliaguet and Porretta, Fondazione CIME/CIME Foundation Subseries.

\bibitem{AT}
S.~N. Armstrong and H.~V. Tran.
\newblock Viscosity solutions of general viscous {H}amilton-{J}acobi equations.
\newblock {\em Math. Ann.}, 361(3-4):647--687, 2015.

\bibitem{BB05}
S.~Bianchini and A.~Bressan.
\newblock Vanishing viscosity solutions of nonlinear hyperbolic systems.
\newblock {\em Ann. of Math. (2)}, 161(1):223--342, 2005.

\bibitem{bre99}
Y.~Brenier.
\newblock Minimal geodesics on groups of volume-preserving maps and generalized
  solutions of the {E}uler equations.
\newblock {\em Comm. Pure Appl. Math.}, 52(4):411--452, 1999.

\bibitem{BHY12}
A.~Bressan, F.~Huang, Y.~Wang, and T.~Yang.
\newblock On the convergence rate of vanishing viscosity approximations for
  nonlinear hyperbolic systems.
\newblock {\em SIAM J. Math. Anal.}, 44(5):3537--3563, 2012.

\bibitem{BY04}
A.~Bressan and T.~Yang.
\newblock On the convergence rate of vanishing viscosity approximations.
\newblock {\em Comm. Pure Appl. Math.}, 57(8):1075--1109, 2004.

\bibitem{CCG21}
S.~Cacace, F.~Camilli, and A.~Goffi.
\newblock A policy iteration method for mean field games.
\newblock {\em ESAIM Control Optim. Calc. Var.}, 27:Paper No. 85, 19, 2021.

\bibitem{C22}
F.~Camilli.
\newblock A policy iteration method for mean field games.
\newblock {\em IFAC-PapersOnLine}, 55(30):406--411, 2022.

\bibitem{CT22}
F.~Camilli and Q.~Tang.
\newblock Rates of convergence for the policy iteration method for mean field
  games systems.
\newblock {\em J. Math. Anal. Appl.}, 512(1):Paper No. 126138, 18, 2022.

\bibitem{cannarsa2018existence}
P.~Cannarsa and R.~Capuani.
\newblock Existence and uniqueness for mean field games with state constraints.
\newblock In {\em P{DE} models for multi-agent phenomena}, volume~28 of {\em
  Springer INdAM Ser.}, pages 49--71. Springer, Cham, 2018.

\bibitem{semiconcave}
P.~Cannarsa and C.~Sinestrari.
\newblock {\em Semiconcave functions, {H}amilton-{J}acobi equations, and
  optimal control}, volume~58 of {\em Progress in Nonlinear Differential
  Equations and their Applications}.
\newblock Birkh\"{a}user Boston, Inc., Boston, MA, 2004.

\bibitem{CSZ17}
F.~Caravenna, R.~Sun, and N.~Zygouras.
\newblock Universality in marginally relevant disordered systems.
\newblock {\em Ann. Appl. Probab.}, 27(5):3050--3112, 2017.

\bibitem{note}
P.~Cardaliaguet.
\newblock Notes on mean field games.
\newblock 2013.
\newblock Available at
  \url{https://www.ceremade.dauphine.fr/~cardaliaguet/MFG20130420.pdf}.

\bibitem{car}
P.~Cardaliaguet.
\newblock Weak solutions for first order mean field games with local coupling.
\newblock In {\em Analysis and geometry in control theory and its
  applications}, volume~11 of {\em Springer INdAM Ser.}, pages 111--158.
  Springer, Cham, 2015.

\bibitem{CG}
P.~Cardaliaguet and P.~J. Graber.
\newblock Mean field games systems of first order.
\newblock {\em ESAIM Control Optim. Calc. Var.}, 21(3):690--722, 2015.

\bibitem{CGPT}
P.~Cardaliaguet, P.~J. Graber, A.~Porretta, and D.~Tonon.
\newblock Second order mean field games with degenerate diffusion and local
  coupling.
\newblock {\em NoDEA Nonlinear Differential Equations Appl.}, 22(5):1287--1317,
  2015.

\bibitem{2012long}
P.~Cardaliaguet, J.-M. Lasry, P.-L. Lions, and A.~Porretta.
\newblock Long time average of mean field games.
\newblock {\em Netw. Heterog. Media}, 7(2):279--301, 2012.

\bibitem{carmona}
R.~Carmona.
\newblock Applications of mean field games in financial engineering and
  economic theory.
\newblock 2020.
\newblock arXiv:2012.05237.

\bibitem{CD20}
S.~Chatterjee and A.~Dunlap.
\newblock Constructing a solution of the {$(2+1)$}-dimensional {KPZ} equation.
\newblock {\em Ann. Probab.}, 48(2):1014--1055, 2020.

\bibitem{Corwin12}
I.~Corwin.
\newblock The {K}ardar-{P}arisi-{Z}hang equation and universality class.
\newblock {\em Random Matrices Theory Appl.}, 1(1):1130001, 76, 2012.

\bibitem{two}
M.~G. Crandall and P.-L. Lions.
\newblock Two approximations of solutions of {H}amilton-{J}acobi equations.
\newblock {\em Math. Comp.}, 43(167):1--19, 1984.

\bibitem{dibenedettobook}
E.~DiBenedetto.
\newblock {\em Degenerate parabolic equations}.
\newblock Universitext. Springer-Verlag, New York, 1993.

\bibitem{djehiche}
B.~Djehiche, A.~Tcheukam, and H.~Tembine.
\newblock Mean-field-type games in engineering.
\newblock {\em AIMS Electronics and Electrical Engineering}, 1(1):18--73, 2017.

\bibitem{evans2010adjoint}
L.~C. Evans.
\newblock Adjoint and compensated compactness methods for {H}amilton-{J}acobi
  {PDE}.
\newblock {\em Arch. Ration. Mech. Anal.}, 197(3):1053--1088, 2010.

\bibitem{ferreira2021existence}
R.~Ferreira, D.~Gomes, and T.~Tada.
\newblock Existence of weak solutions to time-dependent mean-field games.
\newblock {\em Nonlinear Anal.}, 212:Paper No. 112470, 31, 2021.

\bibitem{fetecau2022zero}
R.~C. Fetecau, H.~Huang, D.~Messenger, and W.~Sun.
\newblock Zero-diffusion limit for aggregation equations over bounded domains.
\newblock {\em Discrete Contin. Dyn. Syst.}, 42(10):4905--4936, 2022.

\bibitem{fetecau2019swarming}
R.~C. Fetecau, M.~Kovacic, and I.~Topaloglu.
\newblock Swarming in domains with boundaries: approximation and regularization
  by nonlinear diffusion.
\newblock {\em Discrete Contin. Dyn. Syst. Ser. B}, 24(4):1815--1842, 2019.

\bibitem{GLT21}
P.~Gaudreau~Lamarre, Y.~Lin, and L.-C. Tsai.
\newblock {KPZ} equation with a small noise, deep upper tail and limit shape.
\newblock 2021.
\newblock arXiv:2106.13313.

\bibitem{GPS2}
D.~A. Gomes, E.~Pimentel, and H.~S\'{a}nchez-Morgado.
\newblock Time-dependent mean-field games in the superquadratic case.
\newblock {\em ESAIM Control Optim. Calc. Var.}, 22(2):562--580, 2016.

\bibitem{GPS1}
D.~A. Gomes, E.~A. Pimentel, and H.~S\'{a}nchez-Morgado.
\newblock Time-dependent mean-field games in the subquadratic case.
\newblock {\em Comm. Partial Differential Equations}, 40(1):40--76, 2015.

\bibitem{gra}
P.~J. Graber.
\newblock Optimal control of first-order {H}amilton-{J}acobi equations with
  linearly bounded {H}amiltonian.
\newblock {\em Appl. Math. Optim.}, 70(2):185--224, 2014.

\bibitem{GM}
P.~J. Graber and A.~R. M\'{e}sz\'{a}ros.
\newblock Sobolev regularity for first order mean field games.
\newblock {\em Ann. Inst. H. Poincar\'{e} C Anal. Non Lin\'{e}aire},
  35(6):1557--1576, 2018.

\bibitem{GM19}
P.~J. Graber, A.~R. M\'{e}sz\'{a}ros, F.~J. Silva, and D.~Tonon.
\newblock The planning problem in mean field games as regularized mass
  transport.
\newblock {\em Calc. Var. Partial Differential Equations}, 58(3):Paper No. 115,
  28, 2019.

\bibitem{griffin2022variational}
M.~Griffin-Pickering and A.~R. M{\'e}sz{\'a}ros.
\newblock A variational approach to first order kinetic mean field games with
  local couplings.
\newblock {\em Comm. Partial Differential Equations}, 47(10):1945--2022, 2022.

\bibitem{HDDC20b}
K.~Huang, X.~Di, Q.~Du, and X.~Chen.
\newblock A game-theoretic framework for autonomous vehicles velocity control:
  bridging microscopic differential games and macroscopic mean field games.
\newblock {\em Discrete Contin. Dyn. Syst. Ser. B}, 25(12):4869--4903, 2020.

\bibitem{HDDC20}
K.~Huang, X.~Di, Q.~Du, and X.~Chen.
\newblock Scalable traffic stability analysis in mixed-autonomy using continuum
  models.
\newblock {\em Transp. Res. C: Emerg. Technol.}, 111:616--630, 2020.

\bibitem{huang2007large}
M.~Huang, P.~E. Caines, and R.~P. Malham\'{e}.
\newblock Large-population cost-coupled {LQG} problems with nonuniform agents:
  individual-mass behavior and decentralized {$\epsilon$}-{N}ash equilibria.
\newblock {\em IEEE Trans. Automat. Control}, 52(9):1560--1571, 2007.

\bibitem{HMC}
M.~Huang, R.~P. Malham\'{e}, and P.~E. Caines.
\newblock Large population stochastic dynamic games: closed-loop
  {M}c{K}ean-{V}lasov systems and the {N}ash certainty equivalence principle.
\newblock {\em Commun. Inf. Syst.}, 6(3):221--251, 2006.

\bibitem{16}
H.~Ishii.
\newblock On the equivalence of two notions of weak solutions, viscosity
  solutions and distribution solutions.
\newblock {\em Funkcial. Ekvac}, 38(1):101--120, 1995.

\bibitem{jovanovic}
B.~Jovanovic and R.~W. Rosenthal.
\newblock Anonymous sequential games.
\newblock {\em J. Math. Econom.}, 17(1):77--87, 1988.

\bibitem{KAS16}
P.~Kachroo, S.~Agarwal, and S.~Sastry.
\newblock Inverse problem for non-viscous mean field control: example from
  traffic.
\newblock {\em IEEE Trans. Automat. Control}, 61(11):3412--3421, 2016.

\bibitem{KMS16}
A.~Kamenev, B.~Meerson, and P.~V. Sasorov.
\newblock Short-time height distribution in the one-dimensional
  {K}ardar-{P}arisi-{Z}hang equation: starting from a parabola.
\newblock {\em Phys. Rev. E}, 94(3):032108, 9, 2016.

\bibitem{LL06a}
J.-M. Lasry and P.-L. Lions.
\newblock Jeux \`a champ moyen. {I}. {L}e cas stationnaire.
\newblock {\em C. R. Math. Acad. Sci. Paris}, 343(9):619--625, 2006.

\bibitem{LL06b}
J.-M. Lasry and P.-L. Lions.
\newblock Jeux \`a champ moyen. {II}. {H}orizon fini et contr\^{o}le optimal.
\newblock {\em C. R. Math. Acad. Sci. Paris}, 343(10):679--684, 2006.

\bibitem{LL07}
J.-M. Lasry and P.-L. Lions.
\newblock Mean field games.
\newblock {\em Jpn. J. Math.}, 2(1):229--260, 2007.

\bibitem{lin}
C.-T. Lin and E.~Tadmor.
\newblock {$L^1$}-stability and error estimates for approximate
  {H}amilton-{J}acobi solutions.
\newblock {\em Numer. Math.}, 87(4):701--735, 2001.

\bibitem{LT21st}
Y.~Lin and L.-C. Tsai.
\newblock Short time large deviations of the {KPZ} equation.
\newblock {\em Comm. Math. Phys.}, 386(1):359--393, 2021.

\bibitem{LT22}
Y.~Lin and L.-C. Tsai.
\newblock A lower-tail limit in the weak noise theory.
\newblock 2022.
\newblock arXiv:2210.05629.

\bibitem{MU18}
J.~Magnen and J.~Unterberger.
\newblock The scaling limit of the {KPZ} equation in space dimension 3 and
  higher.
\newblock {\em J. Stat. Phys.}, 171(4):543--598, 2018.

\bibitem{MKV16}
B.~Meerson, E.~Katzav, and A.~Vilenkin.
\newblock Large deviations of surface height in the {K}ardar-{P}arisi-{Z}hang
  equation.
\newblock {\em Phys. Rev. Lett.}, 116(7):070601, 5, 2016.

\bibitem{classical}
S.~Mu\~{n}oz.
\newblock Classical and weak solutions to local first-order mean field games
  through elliptic regularity.
\newblock {\em Ann. Inst. H. Poincar\'{e} C Anal. Non Lin\'{e}aire},
  39(1):1--39, 2022.

\bibitem{orrieri}
C.~Orrieri, A.~Porretta, and G.~Savar\'{e}.
\newblock A variational approach to the mean field planning problem.
\newblock {\em J. Funct. Anal.}, 277(6):1868--1957, 2019.

\bibitem{por14}
A.~Porretta.
\newblock On the planning problem for the mean field games system.
\newblock {\em Dyn. Games Appl.}, 4(2):231--256, 2014.

\bibitem{POR}
A.~Porretta.
\newblock Weak solutions to {F}okker-{P}lanck equations and mean field games.
\newblock {\em Arch. Ration. Mech. Anal.}, 216(1):1--62, 2015.

\bibitem{porretta2023regularizing}
A.~Porretta.
\newblock Regularizing effects of the entropy functional in optimal transport
  and planning problems.
\newblock {\em J. Funct. Anal.}, 284(3):Paper No. 109759, 65, 2023.

\bibitem{S17}
A.~Prosinski and F.~Santambrogio.
\newblock Global-in-time regularity via duality for congestion-penalized mean
  field games.
\newblock {\em Stochastics}, 89(6-7):923--942, 2017.

\bibitem{Quas12}
J.~Quastel.
\newblock Introduction to {KPZ}.
\newblock In {\em Current Developments in Mathematics, 2011}, pages 125--194.
  Int. Press, Somerville, MA, 2012.

\bibitem{tang2023policy}
W.~Tang, H.~V. Tran, and Y.~P. Zhang.
\newblock Policy iteration for the deterministic control problems--a viscosity
  approach.
\newblock 2023.
\newblock arXiv:2301.00419.

\bibitem{tran2011adjoint}
H.~V. Tran.
\newblock Adjoint methods for static {H}amilton-{J}acobi equations.
\newblock {\em Calc. Var. Partial Differential Equations}, 41(3-4):301--319,
  2011.

\bibitem{tran}
H.~V. Tran.
\newblock {\em Hamilton-{J}acobi equations---theory and applications}, volume
  213 of {\em Graduate Studies in Mathematics}.
\newblock American Mathematical Society, Providence, RI, 2021.

\bibitem{Tsai22}
L.-C. Tsai.
\newblock Integrability in the weak noise theory.
\newblock 2022.
\newblock arXiv:2204.00614.

\bibitem{zhang}
Y.~Zhang.
\newblock On continuity equations in space-time domains.
\newblock {\em Discrete Contin. Dyn. Syst.}, 38(10):4837--4873, 2018.

\end{thebibliography}
\bibliographystyle{abbrv}

\end{document}